\documentclass[11pt,letterpaper]{article}
\usepackage{thmtools}
\usepackage{thm-restate}
\usepackage{bm}
\usepackage{mathrsfs}           %
\usepackage{vmargin,fancyhdr}   %
\usepackage{algorithm}
\usepackage{algorithmic}

\usepackage{enumerate}

\usepackage{amsmath,amssymb, amsthm}    %
\usepackage{verbatim}           %
\usepackage{xspace}             %
\usepackage{graphicx,float}     %
\usepackage{ifthen,calc}        %
\usepackage{textcomp}           %
\usepackage{fancybox}           %
\usepackage{hhline}             %
\usepackage{float}              %

\usepackage[vflt]{floatflt}     %
\usepackage[small,compact]{titlesec}
\usepackage{setspace}
\usepackage{subfigure}

\usepackage{color}
\definecolor{ForestGreen}{rgb}{0.1333,0.5451,0.1333}
\usepackage{thumbpdf}
\usepackage[letterpaper,
colorlinks,linkcolor=ForestGreen,citecolor=ForestGreen,
backref,
bookmarks,bookmarksopen,bookmarksnumbered]
{hyperref}

\usepackage{cleveref}

\setpapersize{USletter}
\setmarginsrb{1in}{.5in}        %
{1in}{.5in}        %
{.25in}{.25in}     %
{.25in}{.5in}      %
\newcommand{\showccc}[0]{0}
\newcommand{\ccc}[2][nothing]{%
	\ifthenelse{\showccc=0}{}{
		\ensuremath{^{\Lsh\Rsh}}\marginpar{\raggedright\tiny\textsf{%
				\ifthenelse{\equal{#1}{nothing}}{}{\textbf{#1}\\}#2}}}}

\newcounter{hours}\newcounter{minutes}
\newcommand{\hhmm}{%
	\setcounter{hours}{\time/60}%
	\setcounter{minutes}{\time-\value{hours}*60}%
	\ifthenelse{\value{hours}<10}{0}{}\thehours:%
	\ifthenelse{\value{minutes}<10}{0}{}\theminutes}
\ifthenelse{\showccc=0}{\rhead{}}{\rhead{\today \ [\hhmm]}}
\newtheorem{theorem}{Theorem}

\newtheorem{proposition}{Proposition}
\newtheorem{corollary}{Corollary}
\newtheorem{definition}{Definition}

\newtheorem{lemma}{Lemma}

\newtheorem{problem}{Problem}

\newcommand{\defeq}{:=}
\newcommand{\norm}[1]{\left\lVert#1\right\rVert}
\newcommand{\inprod}[2]{\left\langle#1, #2\right\rangle}
\newcommand{\eps}{\epsilon}
\newcommand{\argmax}{\textup{argmax}}
\newcommand{\argmin}{\textup{argmin}} 
\newcommand{\R}{\mathbb{R}}

\newcommand{\diag}[1]{\textbf{\textup{diag}}\left(#1\right)}
\newcommand{\half}{\frac{1}{2}}
\newcommand{\thalf}{\tfrac{1}{2}}

\newcommand{\E}{\mathbb{E}}

\newcommand{\xset}{\mathcal{X}}
\newcommand{\yset}{\mathcal{Y}}
\newcommand{\zset}{\mathcal{Z}}
\newcommand{\x}{^x}
\newcommand{\y}{^y}
\definecolor{burntorange}{rgb}{0.8, 0.33, 0.0}

\newcommand{\tO}{\tilde{O}}

\newcommand{\hx}{x_{t + \half}}
\newcommand{\hy}{y_{t + \half}}
\newcommand{\hv}{v_{t + \half}}
\newcommand{\px}{x_{t + 1}}

\newcommand{\pv}{v_{t + 1}}

\newcommand{\hxi}{x_{t + \half}^{(i)}}
\newcommand{\di}{\Delta_{t}^{(i)}}

\newcommand{\hvi}{\hv}
\newcommand{\wi}{w_t^{(i)}}
\newcommand{\pxi}{x_{t + 1}^{(i)}}

\newcommand{\pvi}{v_{t + 1}^{(i)}}
\newcommand{\pzi}{z_{t + 1}^{(i)}}
\newcommand{\mux}{\mu_x}
\newcommand{\muy}{\mu_y}
\newcommand{\lxx}{L_{xx}}
\newcommand{\lxy}{L_{xy}}
\newcommand{\lyy}{L_{yy}}
\newcommand{\Par}[1]{\left(#1\right)}
\newcommand{\Brack}[1]{\left[#1\right]}
\newcommand{\Brace}[1]{\left\{#1\right\}}
\newcommand{\lam}{\lambda}
\newcommand{\Prox}{\textup{Prox}}
\newcommand{\jac}{\mathbf{J}}
  \usepackage{nth}
  \usepackage{intcalc}

\usepackage{url}
\usepackage{enumitem}

\begin{document}

\begin{titlepage}
	\def\thepage{}
	\thispagestyle{empty}
	
	\title{Relative Lipschitzness in Extragradient Methods and \\ a Direct Recipe for Acceleration} 
	
	\date{}
	\author{
		Michael B. Cohen\thanks{MIT, {\tt micohen@mit.edu}}
		\and
		Aaron Sidford\thanks{Stanford University, {\tt \{sidford, kjtian\}@stanford.edu}}
		\and
		Kevin Tian\footnotemark[2] 
	}
	\maketitle

\abstract{
We show that standard extragradient methods (i.e.\ mirror prox \cite{Nemirovski04} and dual extrapolation \cite{Nesterov07}) recover optimal accelerated rates for first-order minimization of smooth convex functions. To obtain this result we provide a fine-grained characterization of the convergence rates of extragradient methods for solving monotone variational inequalities in terms of a natural condition we call \emph{relative Lipschitzness}. We further generalize this framework to handle local and randomized notions of relative Lipschitzness and thereby recover rates for box-constrained $\ell_\infty$ regression based on area convexity \cite{Sherman17} and complexity bounds achieved by accelerated (randomized) coordinate descent \cite{ZhuQRY16, NesterovS17} for smooth convex function minimization.
}
 	
\end{titlepage}
\pagenumbering{gobble}
\newpage
\setcounter{tocdepth}{2}
{
	\hypersetup{linkcolor=black}
	\tableofcontents
}
\newpage
\pagenumbering{arabic}
\section{Introduction}
\label{sec:intro}

We study the classic extragradient algorithms of mirror prox \cite{Nemirovski04} and dual extrapolation \cite{Nesterov07} for solving variational inequalities (VIs) in monotone operators. 
This family of problems includes convex optimization and finding the saddle point of a convex-concave game. Due to applications of the latter to adversarial and robust training, extragradient methods have received significant recent attention in the machine learning community, see e.g.\ \cite{ChavdarovaGFL19, MertikopoulosLZ19, HsiehIMM19}. Further, extragradient methods have been the subject of increasing study by the theoretical computer science and optimization communities due to recent extragradient-based
runtime improvements for problems including maximum flow \cite{Sherman17} and zero-sum games \cite{CarmonJST19, CarmonJST20}.

Given a Lipschitz monotone operator and a bounded strongly-convex regularizer, mirror prox \cite{Nemirovski04} and dual extrapolation \cite{Nesterov07} achieve $O(T^{-1})$ regret for solving the associated VI after $T$ iterations. This rate is worst-case optimal when the Lipschitzness of the operator and strong convexity of the regularizer are with respect to the Euclidean norm \cite{OuyangX19}. However, in certain structured problems related to VIs, alternative analyses and algorithms can yield improved rates. For instance, when minimizing a smooth convex function (i.e. one with a Lipschitz gradient), it is known that accelerated rates of $O(T^{-2})$ are attainable, improving upon the standard $O(T^{-1})$ extragradient rate for the naive associated VI (see \Cref{sec:naive_smooth_rate}). Further, algorithms inspired by extragradient methods have been developed recovering the $O(T^{-2})$ rate \cite{DiakonikolasO18, WangA18}.

Additionally, alternative analyses of extragradient methods, such as optimism \cite{RakhlinS13} and area convexity \cite{Sherman17} have shown that under assumptions beyond a Lipschitz operator and a strongly convex regularizer, improved rates can be achieved. These works leveraged modified algorithms which run efficiently under such non-standard assumptions. Further, the area convexity-based methods of \cite{Sherman17} have had a number of implications, including faster algorithms for $\ell_\infty$ regression, maximum flow and multicommodity flow \cite{Sherman17} as well as improved parallel algorithms for work-efficient positive linear programming \cite{BoobSW19} and optimal transport \cite{JambulapatiST19}.

In this work we seek a better understanding of these structured problems and the somewhat disparate-seeming analyses and algorithms for solving them. We ask, \emph{are the algorithmic changes enabling these results necessary? Can standard mirror prox and dual extrapolation be leveraged to obtain these results? Can we unify analyses for these problems, and further clarify the relationship between acceleration, extragradient methods, and primal-dual methods?}

Towards addressing these questions, we provide a general condition, which we term \emph{relative Lipschitzness} (cf.\ Definition~\ref{def:spc}), and analyze the convergence of mirror prox and dual extrapolation for a monotone relatively Lipschitz operator.\footnote{A somewhat similarly-named property appeared in \cite{Lu19}, which also studied mirror descent algorithms under relaxed conditions; their property $\norm{g(x)}_*^2\le \frac{MV_x(y)}{\norm{y - x}^2}$ for all $x, y$, is different than the one we propose. Further, during the preparation of this work, the relative Lipschitzness condition we propose was also independently stated in \cite{StonyakinaTGDADPAP20}. However, the work \cite{StonyakinaTGDADPAP20} does not derive the various consequences of relative Lipschitzness contained in this work (e.g.\ recovery of acceleration and randomized acceleration, as well as applications of area convexity).} This condition is derived directly from the standard analysis of the methods and is stated in terms of a straightforward relationship between the operator $g$ and the regularizer $r$ which define the algorithm. Our condition is inspired by both area convexity as well as the ``relative smoothness'' condition in convex optimization \cite{BauschkeBT17, LuFN18}, and can be thought of as a generalization of the latter to variational inequalities (see \Cref{lem:rel_lipschitz}). Further, through this analysis we show that standard extragradient methods directly yield accelerated rates for smooth minimization and recover the improved rates of \cite{Sherman17} for box-constrained $\ell_\infty$ regression, making progress on the questions outlined above. We also show our methods recover certain randomized accelerated rates and have additional implications, outlined below.

\textbf{Extragradient methods directly yield acceleration.} In Section~\ref{sec:acceleration}, we show that applying algorithms of \cite{Nemirovski04, Nesterov07} to a minimax formulation of $\min_{x \in \R^d} f(x)$, when $f$ is smooth and strongly convex, yields accelerated minimization rates when analyzed via relative Lipschitzness. Specifically, we consider the following problem, termed the \emph{Fenchel game} in \cite{WangA18}:
\begin{equation}\label{eq:pdobj}\min_{x \in \R^d} \max_{y \in \R^d} \inprod{y}{x} - f^*(y),\end{equation}
and show that when $f$ is $\mu$-strongly convex and $L$-smooth, $O(\sqrt{L/\mu})$ iterations of either mirror prox \cite{Nemirovski04} or dual extrapolation \cite{Nesterov07} produces an average iterate which halves the function error of $f$. By repeated application, this yields an accelerated linear rate of convergence and the optimal $O(T^{-2})$ rates for non-strongly convex, smooth function minimization by a reduction \cite{ZhuH16}.  Crucially, to attain this rate we give a sharpened bound on the relative Lipschitzness of the gradient operator of \eqref{eq:pdobj} with respect to a primal-dual regularizer. 

Our result advances a recent line of research, \cite{AbernethyLLW18, AbernethyLW19, DiakonikolasO18}, on applying primal-dual analyses to shed light on the mysterious nature of acceleration. Specifically, \cite{AbernethyLLW18, AbernethyLW19} show that the classical algorithm of \cite{Nesterov83} can be rederived via applying primal-dual ``optimistic'' dynamics, inspired by the framework of \cite{RakhlinS13}. Further,  \cite{DiakonikolasO18} showed that an appropriate discretization of dynamics inspired by extragradient algorithms yields an alternative accelerated algorithm. While these results clarify the primal-dual nature of acceleration, additional tuning is ultimately required to obtain their final algorithms and analysis. We obtain acceleration as a direct application of known frameworks, i.e.\ standard mirror prox and dual extrapolation, applied to the formulation \eqref{eq:pdobj}, and hope this helps demystify acceleration. 

In Appendix~\ref{appendix:generalize}, we further show that analyzing extragradient methods tailored to strongly monotone operators via relative Lipschitzness, and applying this more fine-grained analysis to a variant of the objective \eqref{eq:pdobj}, also yields an accelerated linear rate of convergence. The resulting proof strategy extends readily to accelerated minimization of smooth and strongly convex functions in general norms, as we discuss at the end of Section~\ref{sec:acceleration}, and we believe it may be of independent interest.

Finally, we remark that there has been documented difficulty in accelerating the minimization of relatively smooth functions \cite{HanzelyRX18}; this was also explored more formally by \cite{DragomirTdAB19}. It is noted in \cite{HanzelyRX18}, as well as suggested in others (e.g.\ in the development of area convexity \cite{Sherman17}) that this discrepancy may be due to acceleration fundamentally requiring conditions on relationships between groups of three points, rather than two. Our work, which presents an alternative three-point condition yielding accelerated rates, sheds light on this phenomenon and we believe it is an interesting future direction to explore the relationships between our condition and other alternatives in the literature which are known to yield acceleration.

\textbf{Area convexity for bilinear box-simplex games.} In Section~\ref{sec:areaconvex}, we draw a connection between relative Lipschitzness and the notion of an ``area convex'' regularizer, proposed by \cite{Sherman17}. Area convexity is a property which weakens strong convexity, but is suitable for extragradient algorithms with a linear operator. It was introduced in the context of solving a formulation of approximate undirected maximum flow via box-constrained $\ell_\infty$ regression, or more generally approximating bilinear games between a box variable and a simplex variable. The algorithm of \cite{Sherman17} applied to bilinear games was a variant of standard extragradient methods and analyzed via area convexity, which was proven via solving a subharmonic partial differential equation. We show that mirror prox, as analyzed by a local variant of relative Lipschitzness, yields the same rate of convergence as implied by area convexity, for box-simplex games. Our proof of this rate is straightforward and based on a simple Cauchy-Schwarz argument after demonstrating local stability of iterates. %

\textbf{Randomized extragradient methods via local variance reduction.} In general, the use of stochastic operator estimates in the design of extragradient algorithms for solving general VIs is not as well-understood as their use in the special case of convex function minimization. The best-known known stochastic methods for solving VIs \cite{JuditskyNT11} with bounded-variance stochastic estimators obtain $O(T^{-1/2})$ rates of convergence; this is by necessity, from known classical lower bounds on the rate of the special case of  stochastic convex optimization \cite{NemirovskiY83}. 
 Towards advancing the randomized extragradient toolkit, we ask: when can improved $O(T^{-1})$ rates of convergence be achieved by stochastic algorithms for solving specific VIs and fine-grained bounds on estimator variance (i.e.\ more local notions of variance)? This direction is inspired by analogous results in convex optimization, where reduced-variance and accelerated rates have been obtained, matching and improving upon their deterministic counterparts \cite{Johnson013, SchmidtRB17, Zhu17, LeeS13, NesterovS17, ZhuQRY16}.  

For the special case of bilinear games, this question was recently addressed by the works \cite{PalaniappanB16, CarmonJST19}, using proximal reductions to attain improved rates. In this work, we give a framework for direct stochastic extragradient method design bypassing the variance bottleneck limiting prior algorithms to a $O(T^{-1/2})$ rate of convergence for problems with block-separable structure. We identify a particular criterion of randomized operators used in the context of extragradient algorithms (cf.\ Proposition~\ref{prop:generalmp}) which enables $O(T^{-1})$ rates of convergence. Our approach is a form of ``local variance reduction'', where estimators in an iteration of the method share a random seed and we take expectations over the whole iteration in the analysis. Our improved estimator design exploits the separable structure of the problem; it would be interesting to design a more general variance reduction framework for randomized extragradient methods.

Formally, we apply our local variance reduction framework in Section~\ref{sec:coordinate} to show that an instance of our new randomized extragradient methods recover acceleration for coordinate-smooth functions, matching the known tight rates of \cite{ZhuQRY16, NesterovS17}. Along the way, we give a variation of relative Lipschitzness capturing an analagous property between a locally variance-reduced randomized gradient estimator and a regularizer, which we exploit to obtain our runtime. We note that a similar approach was taken in \cite{SidfordT18} to obtain faster approximate maximum flow algorithms in the bilinear minimax setting; here, we generalize this strategy and give conditions under which our variance reduction technique obtains improved rates for extragradient methods more broadly.

\textbf{Additional contributions.} A minor contribution of our work is that we show (Appendix~\ref{app:minimax}) that relative Lipschitzness implies new rates for minimax convex-concave optimization, taking a step towards closing the gap with lower bounds with \emph{fine-grained} dependence on problem parameters. Under operator-norm bounds on blocks of the Hessian of a convex-concave function, as well as blockwise strong convexity assumptions, \cite{ZhangHZ19} showed a lower bound on the convergence rate to obtain an $\eps$-approximate saddle point. When the blockwise operator norms of the Hessian are roughly equal, \cite{LinJJ20} gave an algorithm matching the lower bound up to a polylogarithmic factor, using an alternating scheme repeatedly calling an accelerated proximal point reduction. Applying our condition with a strongly monotone variant of the mirror prox algorithm of \cite{Nemirovski04} yields a new fine-grained rate for minimax optimization, improving upon the runtime of \cite{LinJJ20} for a range of parameters. Our algorithm is simple and the analysis follows directly from a tighter relative Lipschitzness bound; we note the same result can also be obtained by considering an operator norm bound of the problem after a rescaling of space, but we include this computation because it is a straightforward implication of our condition.

Finally, in  Appendix~\ref{appendix:area_convex}, we discuss the relation of relative Lipschitzness to another framework for analyzing extragradient methods: namely, we note that our proof of the sufficiency of relative Lipschitzness recovers known bounds for optimistic mirror descent \cite{RakhlinS13}. 
\section{Notation}
\label{sec:definitions}

\textbf{General notation.}
Variables are in $\mathbb{R}^d$ unless otherwise noted. $e_i$ is the $i^{th}$ standard basis vector. $\norm{\cdot}$ denotes an arbitrary norm; the dual norm is $\norm{\cdot}_*$, defined as $\norm{x}_* \defeq \max_{\norm{y} \leq 1} y^\top x$. For a variable on two blocks $z \in \xset \times \yset$, we refer to the blocks by $z\x$ and $z\y$. We denote the domain of $f: \R^d \rightarrow \R$ by $\xset$; when unspecified, $\xset = \R^d$. When $f$ is clear from context, $x^*$ is any minimizing argument. We call any $x$ with $f(x) \leq f(x^*) + \epsilon$ an \emph{$\epsilon$-approximate minimizer}. 

\textbf{Bregman divergences.}
The Bregman divergence induced by distance-generating convex $r$ is
\begin{equation*}
V^r_x(y) \defeq r(y) - r(x) - \inprod{\nabla r(x)}{y - x}.
\end{equation*}
The Bregman divergence is always nonnegative, and convex in its argument. We define the following proximal operation with respect to a divergence from point $z$.
\begin{equation}\label{eq:proxdef}
\textup{Prox}^r_{x}(g) \defeq \textup{argmin}_{y} \left\{\inprod{g}{y} + V^r_x(y)\right\}.
\end{equation}

\textbf{Functions.}
We say $f$ is $L$-smooth in $\norm{\cdot}$ if $\norm{\nabla f(x) - \nabla f(y)}_* \leq L \norm{x - y}$, or equivalently $f(y) \leq f(x) + \inprod{\nabla f(x)}{y - x} + \frac{L}{2}\norm{y - x}^2$ for $x, y \in \xset$. If $f$ is twice-differentiable, equivalently $y^\top \nabla^2 f(x) y \leq L \norm{y}^2$. We say differentiable $f$ is $\mu$-strongly convex if for some $\mu \geq 0$, $f(y) \geq f(x) + \inprod{\nabla f(x)}{y - x} + \frac{\mu}{2}\|y - x\|^2$ for $x, y \in \xset$. We also say $f$ is $\mu$-strongly convex with respect to a distance-generating function $r$ if $V^f_x(y) \ge \mu V^r_x(y)$ for all $x, y \in \xset$. Further, we use standard results from convex analysis throughout, in particular facts about Fenchel duality, and defer these definitions and proofs to Appendix~\ref{appendix:convex_analysis}.

\textbf{Saddle points.}
We call function $h(x, y)$ of two variables \emph{convex-concave} if its restrictions to $x$ and $y$ are convex and concave respectively. We call $(x, y)$ an \emph{$\epsilon$-approximate saddle point} if $\max_{y'} \{h(x, y')\} - \min_{x'} \{h(x', y) \} \leq \epsilon.$ We equip any differentiable convex-concave function with gradient operator $g(x, y) \defeq (\nabla_x h(x, y), -\nabla_y h(x, y))$. 

\textbf{Monotone operators.} We call operator $g: \zset \rightarrow \zset^*$ monotone if $\inprod{g(w) - g(z)}{w - z} \ge 0$ for all $w$, $z\in \zset$. Examples include  the gradient of a convex function and gradient operator of a convex-concave function. We call $g$ $m$-strongly monotone with respect to $r$ if $\inprod{g(w) - g(z)}{w - z} \ge m\Par{V^r_w(z) + V^r_z(w)} = m\inprod{\nabla r(w) - \nabla r(z)}{w - z}$. We call $z^* \in \zset$ the solution to the variational inequality (VI) in a monotone operator $g$ if $\inprod{g(z^*)}{z^* - z} \le 0$ for all $z \in \zset$.\footnote{This is also known as a ``strong solution''. A ``weak solution'' is a $z^*$ with $\inprod{g(z)}{z^* - z} \le 0$ for all $z$.} Examples include the minimizer of a convex function and the saddle point of a convex-concave function.
\section{Extragradient convergence under relative Lipschitzness}
\label{sec:overview}

\begin{algorithm}
	\caption{$\textsc{Mirror-Prox}(z_0, T)$: Mirror prox \cite{Nemirovski04}}
	\begin{algorithmic}\label{alg:mp}
		\STATE \textbf{Input:} Distance generating $r$, $\lam$-relatively Lipschitz monotone $g: \zset \rightarrow \zset^*$, initial point $z_0 \in \zset$
		\FOR {$0 \le t < T$}
		\STATE $w_t \gets \Prox^r_{z_t}(\tfrac{1}{\lam}g(z_t))$
		\STATE $z_{t + 1} \gets \Prox^r_{z_t}(\tfrac{1}{\lam}g(w_t))$
		\ENDFOR
	\end{algorithmic}
\end{algorithm}

We give a brief presentation of mirror prox \cite{Nemirovski04}, and a convergence analysis under relative Lipschitzness. Our results also hold for dual extrapolation \cite{Nesterov07}, which can be seen as a ``lazy'' version of mirror prox updating a state in dual space (see \cite{Bubeck15}); we defer details to Appendix~\ref{appendix:dual_extrapolation}.

\begin{definition}[Relative Lipschitzness]
	\label{def:spc}
	For convex $r:\zset \rightarrow \R$, we call operator $g: \zset \rightarrow \zset^*$ \emph{$\lam$-relatively Lipschitz with respect to $r$} if for every three $z, w, u \in \zset$,
	\begin{equation*}
	\inprod{g(w) - g(z)}{w - u} \leq \lam \left(V^r_z(w) + V^r_w(u)\right)
	\end{equation*}
\end{definition}

Definition~\ref{def:spc} can be thought of as an alternative to a natural nonlinear analog of the area convexity condition of \cite{Sherman17} displayed below:
\[\inprod{g(w) - g(z)}{w - u} \le \lam\Par{r(z) + r(w) - r(u) - 3r\Par{\frac{z + w + u}{3}}}.\] 
Our proposed alternative is well-suited for the standard analyses of extragradient methods such as mirror prox and dual extrapolation. For the special case of bilinear minimax problems in a matrix $A$, the left hand side of Definition~\ref{def:spc} measures the area of a triangle in a geometry induced by $A$.

Relative Lipschitzness encapsulates the more standard assumptions that $g$ is Lipschitz and $r$ is strongly convex (\Cref{lem:basiclam}), as well as the more recent assumptions that $f$ is convex and relatively smooth with respect to $r$ \cite{BauschkeBT17, LuFN18} (\Cref{lem:rel_lipschitz}). We defer proofs to Appendix~\ref{appendix:convex_analysis}.

\begin{restatable}{lemma}{restatebasiclam}
\label{lem:basiclam}
	If $g$ is $L$-Lipschitz and $r$ is $\mu$-strongly convex in $\norm{\cdot}$, $g$ is $L/\mu$-relatively Lipschitz with respect to $r$.
\end{restatable}

\begin{restatable}{lemma}{restaterelsmooth}
\label{lem:rel_lipschitz}
If $f$ is $L$-relatively smooth with respect to $r$, i.e. $V^f_x(y) \leq L V^r_x(y)$ for all $x$ and $y$, then $g$, defined by $g(x) \defeq \nabla f(x)$ for all $x$, is $L$-relatively Lipschitz with respect to $r$.
\end{restatable}

We now give an analysis of Algorithm~\ref{alg:mp} showing the average ``regret'' $\inprod{g(w_t)}{w_t - u}$ of iterates decays at a $O(T^{-1})$ rate. This strengthens Lemma 3.1 of \cite{Nemirovski04}.

\begin{proposition}
	\label{prop:mirrorprox}
	The iterates $\{w_t\}$ of Algorithm~\ref{alg:mp} satisfy for all $u \in \zset$, 
	\[\sum_{0 \le t < T} \inprod{g(w_t)}{w_t - u} \le \lam V^r_{z_0}(u).\]
\end{proposition}
\begin{proof}
First-order optimality conditions of $w_t$, $z_{t + 1}$ with respect to $u$ imply (see Lemma~\ref{lem:proxprop})
\begin{equation}\label{eq:foopt}
\begin{aligned}
\frac{1}{\lam}\inprod{g(z_t)}{w_t - z_{t + 1}} &\leq V^r_{z_t}(z_{t + 1}) - V^r_{w_t}(z_{t + 1}) - V^r_{z_t}(w_t), \\
\frac{1}{\lam}\inprod{g(w_t)}{z_{t + 1} - u} &\leq V^r_{z_t}(u) - V^r_{z_{t + 1}}(u) - V^r_{z_t}(z_{t + 1}).
\end{aligned}
\end{equation}
Adding and manipulating gives, via relative Lipschitzness (Definition~\ref{def:spc}),
\begin{equation}
\begin{aligned}
\label{eq:telescope}
\frac{1}{\lam}\inprod{g(w_t)}{w_t - u} &\leq V^r_{z_t}(u) - V^r_{z_{t + 1}}(u) + \frac{1}{\lam}\inprod{g(w_t) - g(z_t)}{w_t - z_{t + 1}} - V^r_{w_t}(z_{t + 1}) - V^r_{z_t}(w_t) \\&\leq V^r_{z_t}(u) - V^r_{z_{t + 1}}(u).
\end{aligned}
\end{equation}
Finally, summing and telescoping \eqref{eq:telescope} yields the desired conclusion.
\end{proof}

We briefly comment on how to use Proposition~\ref{prop:mirrorprox} to approximately solve convex-concave games in a function $f(x, y)$. By applying convexity and concavity appropriately to the regret guarantee (and dividing by $T$, the iteration count), one can replace the left hand side of the guarantee with the duality gap of an average point $\bar{w}$ against a point $u$, namely $f(w^x, u^y) - f(u^x, w^y)$. By maximizing the right hand side over $u$, this can be converted into an overall duality gap guarantee. For some of our applications in following sections, $u$ will be some fixed point (rather than a best response) and the regret statement will be used in a more direct manner to prove guarantees. %
\section{Acceleration via relative Lipschitzness}
\label{sec:acceleration}

We show that directly applying Algorithm~\ref{alg:mp} to the optimization problem \eqref{eq:pdobj} recovers an accelerated rate for first-order convex function minimization (for simplicity, we focus on the $\ell_2$ norm here; our methods extend to general norms, discussed in Appendix~\ref{appendix:generalize}). Our main technical result, Lemma~\ref{lem:betterlambda}, shows the gradient operator of \eqref{eq:pdobj} is relatively Lipschitz in the natural regularizer induced by $f$, which combined with Proposition~\ref{prop:mirrorprox} gives our main result, Theorem~\ref{thm:accel}. Crucially, our method regularizes the dual variable with $f^*$, the Fenchel dual of $f$, which we show admits efficient implementation, allowing us to obtain our improved bound on the relative Lipschitzness parameter.

\begin{lemma}[Relative Lipschitzness for the Fenchel game]
	\label{lem:betterlambda}
	Let $f: \R^d \rightarrow \R$ be $L$-smooth and $\mu$-strongly convex in the Euclidean norm $\norm{\cdot}_2$. Let $g(x, y) = \Par{y, \nabla f^*(y) - x}$ be the gradient operator of the convex-concave problem \eqref{eq:pdobj}, and define the distance-generating function $r(x, y) \defeq \tfrac{\mu}{2}\norm{x}_2^2 + f^*(y)$. Then, $g$ is $1 + \sqrt{\tfrac{L}{\mu}}$-relatively Lipschitz with respect to $r$.
\end{lemma}
\begin{proof}
	Consider three points $z = (z\x, z\y),\; w = (w\x, w\y),\; u = (u\x, u\y)$. By direct calculation,
\begin{equation}\label{eq:expand}\inprod{g(w) - g(z)}{w - u} = \inprod{w\y - z\y}{w\x - u\x} + \inprod{-w\x + z\x + \nabla f^*(w\y) - \nabla f^*(z\y)}{w\y - u\y}.\end{equation}
	By Cauchy-Schwarz and $L^{-1}$-strong convexity of $f^*$ (cf. Lemma~\ref{lem:scdual}) respectively, we have
	\begin{equation}
	\label{eq:accelcs}
	\begin{aligned}
	\inprod{w\y - z\y}{w\x - u\x} + \inprod{z\x - w\x}{w\y - u\y} \le \norm{w\y - z\y}_2 \norm{w\x - u\x}_2 + \norm{z\x - w\x}_2 \norm{w\y - u\y}_2\\
	\le \sqrt{\frac{L}{\mu}}\Par{\frac{\mu}{2}\norm{w\x - z\x}_2^2 + \frac{\mu}{2}\norm{w\x - u\x}_2^2 + \frac{1}{2L}\norm{w\y - z\y}_2^2 + \frac{1}{2L}\norm{w\y - u\y}_2^2} \\
	\le \sqrt{\frac{L}{\mu}}\Par{V^r_z(w) + V^r_w(u)}.
	\end{aligned}
	\end{equation}
	The second line used Young's inequality twice. Furthermore, by convexity of $f^*$ from $z\y$ to $u\y$, 
	\begin{equation}
	\label{eq:threeptconvex}
	\begin{aligned}
	\inprod{\nabla f^*(w\y) - \nabla f^*(z\y)}{w\y - u\y} \\
	= \inprod{\nabla f^*(z\y)}{u\y - z\y} - \inprod{\nabla f^*(w\y)}{u\y - w\y} - \inprod{\nabla f^*(z\y)}{w\y - z\y} \\
	\le f^*(u\y) - f^*(z\y) - \inprod{\nabla f^*(w\y)}{u\y - w\y} -  \inprod{\nabla f^*(z\y)}{w\y - z\y} \\
	= V^{f^*}_{z\y}(w\y) + V^{f^*}_{w\y}(u\y) \le V^r_z(w) + V^r_w(u).
	\end{aligned}
	\end{equation}
	The last inequality used separability of $r$ and nonnegativity of divergences. Summing the bounds \eqref{eq:accelcs} and \eqref{eq:threeptconvex} and recalling \eqref{eq:expand} yields the conclusion, where we use Definition~\ref{def:spc}.
\end{proof}

We also state a convenient fact about the form our iterates take.

\begin{lemma}\label{lem:iterateform}
In the setting of Lemma~\ref{lem:betterlambda}, let $z_t = (x_t, y_t)$, $w_t = (\hx, \hy)$ be iterates produced by running Algorithm~\ref{alg:mp} on the pair $g$, $r$. Suppose $y_0 = \nabla f(v_0)$ for some $v_0$. Then, $\hy$ and $y_{t + 1}$ can be recursively expressed as $\hy = \nabla f(\hv)$, $y_{t + 1} = \nabla f(v_{t + 1})$, for
\[\hv \gets v_t + \frac{1}{\lam}(x_t - v_t),\;  v_{t + 1} \gets v_t + \frac{1}{\lam}\Par{\hx - \hv}.\]
\end{lemma}
\begin{proof}
We prove this inductively; consider some iteration $t$. Assuming $y_t = \nabla f(v_t)$, by definition
\begin{align*}\hy &= \argmin_y\Brace{\inprod{\frac{1}{\lam}\Par{\nabla f^*(y_t) - x_t}}{y} + V^{f^*}_{y_t}(y)} \\
&= \argmax_y \Brace{\inprod{\frac{1}{\lam}(x_t - v_t) + v_t}{y} - f^*(y)} = \nabla f\Par{v_t + \frac{1}{\lam}(x_t - v_t)}. \end{align*}
Here, we used standard facts about convex conjugates (see Lemma~\ref{lem:argmaxconj}). A similar argument shows that we can compute implicitly $y_{t + 1} = \nabla f(v_t + \tfrac{1}{\lam}(\hx - \hv))$.
\end{proof}
We now prove Theorem~\ref{thm:accel}, i.e.\ that we can halve function error in $O\Par{\sqrt{\frac L \mu}}$ iterations of Algorithm~\ref{alg:mp}. Simply iterating Theorem~\ref{thm:accel} yields a linear rate of convergence for smooth, strongly convex functions, yielding an $\eps$-approximate minimizer in $O\Par{\sqrt{\frac{L}{\mu}} \log \frac{f(x_0) - f(x^*)}{\eps}}$ iterations.
\begin{theorem}\label{thm:accel}
In the setting of Lemma~\ref{lem:betterlambda}, run $T \ge 4\lam$ iterations of Algorithm~\ref{alg:mp} initialized at $z_0 = (x_0, \nabla f(x_0))$ on the pair $g, r$ with $\lam = 1 + \sqrt{\tfrac{L}{\mu}}$, and define
\[\bar{v} = \frac{1}{T}\sum_{0 \le t < T} v_{t + \half} \text{ where } w_t = \Par{\hx, \nabla f (\hv) }.\] 
Then we have $f(\bar{v}) - f(x^*) \le \thalf(f(x_0) - f(x^*))$, where $x^*$ minimizes $f$.
\end{theorem}
\begin{proof}
First, we remark that this form of $w_t$ follows from Lemma~\ref{lem:iterateform}, and correctness of $\lam$ follows from Lemma~\ref{lem:betterlambda}. By an application of Proposition~\ref{prop:mirrorprox}, letting $u = (x^*, \nabla f(x^*))$,
\begin{align*}\frac{1}{T}\sum_{0 \le t < T} \inprod{g(w_t)}{w_t - u} &\le \frac{\lam}{T} \cdot V^r_{z_0}(u) \le \frac{1}{4}\Par{\frac{\mu}{2}\norm{x_0 - x^*}_2^2 + V^{f^*}_{\nabla f(x_0)}(\nabla f(x^*))}\\
&= \frac{1}{4} \Par{\frac{\mu}{2}\norm{x_0 - x^*}_2^2 + f(x_0) - f(x^*)} \le \half \Par{f(x_0) - f(x^*)}.\end{align*}
The second line used the definition of divergence in $f^*$ (see Lemma~\ref{lem:bregofdual}) and strong convexity of $f$, which implies $f(x_0) \ge f(x^*) + \tfrac{\mu}{2}\norm{x_0 - x^*}_2^2$. Moreover, by the definition of $g$ and $\nabla f(x^*) = 0$,
\begin{align*}\frac{1}{T}\sum_{0 \le t < T} \inprod{g(w_t)}{w_t - u} &= \frac{1}{T}\sum_{0 \le t < T} \inprod{\nabla f(\hv)}{\hx - x^*} + \inprod{\hv - \hx}{\nabla f(\hv) }\\
&\ge \frac{1}{T}\sum_{0 \le t < T} f(\hv) - f(x^*) \ge f(\bar{v}) - f(x^*).\end{align*}
The last line used convexity twice. Combining these two derivations yields the conclusion.
\end{proof}

For convenience, we state the full algorithm of Theorem~\ref{thm:accel} as Algorithm~\ref{alg:accel_basic}.

\begin{algorithm}
	\caption{$\textsc{EG-Accel}(x_0, \epsilon)$: Extragradient accelerated smooth minimization}
	\begin{algorithmic}
		\label{alg:accel_basic}
		\STATE \textbf{Input:} $x_0 \in \R^d$, $f$ $L$-smooth and $\mu$-strongly convex in $\norm{\cdot}_2$, and $\eps_0 \ge f(x_0) - f(x^*)$
		\STATE \textbf{Output:} $\eps$-approximate minimizer of $f$
		\STATE $\lam \gets 1 + \sqrt{L/\mu}$, $x^{(0)} \gets x_0$, $T \gets 4\lceil\lam\rceil$, $K \gets \lceil\log_2\tfrac{\eps_0}{\eps}\rceil$
		\FOR {$0 \le k < K$}
		\STATE $x_0 \gets x^{(k)}$, $v_0 \gets x_0$
		\FOR {$0 \le t < T$}
		\STATE $\hx \gets x_t - \tfrac{1}{\mu\lam}\nabla f(v_t)$ and $\hv \gets v_t + \tfrac{1}{\lam}(x_t - v_t)$
		\STATE $\px \gets x_t - \tfrac{1}{\mu\lam}\nabla f(\hv)$ and $\pv \gets v_t + \tfrac{1}{\lam}(\hx - \hv)$ 
		\ENDFOR
		\STATE $x^{(k + 1)} \gets \frac{1}{T}\sum_{0\le t < T}\hv$
		\ENDFOR
		\RETURN $x^{(K)}$
	\end{algorithmic}
\end{algorithm}

In Appendix~\ref{appendix:generalize}, we give an alternative proof of acceleration leveraging relative Lipschitzness, as well as a variant of extragradient methods suited for strongly monotone operators (cf.\ Appendix~\ref{appendix:dual_extrapolation}), by applying these tools to the saddle point problem (to be contrasted with \eqref{eq:pdobj})
\[\min_{x \in \R^d} f(x) = \min_{x \in \R^d} \max_{y \in \R^d} \frac{\mu}{2}\norm{x}_2^2 + \inprod{y}{x} - h^*(y),\text{ where } h(x) \defeq f(x) - \frac{\mu}{2}\norm{x}_2^2.\] 
This alternative proof strategy readily generalizes the accelerated rate of  Theorem~\ref{thm:accel} to general norms. While the rates attained in Appendix~\ref{appendix:generalize} are slightly less sharp (losing a $\tfrac L \mu$ factor in the logarithm) when compared to Theorem~\ref{thm:accel}, the analysis is arguably simpler. This is in the sense that Appendix~\ref{appendix:generalize} shows a potential function decreases at a linear rate in every iteration, rather than requiring $O(\sqrt{L/\mu})$ iterations to halve it.
\section{Area convexity rates for box-simplex games via relative Lipschitzness}
\label{sec:areaconvex}

In this section, we show that a local variant of Definition~\ref{def:spc} recovers the improved convergence rate achieved by \cite{Sherman17} for box-constrained $\ell_\infty$-regression, and more generally box-simplex bilinear games. Specifically, we will use the following result, a simple extension to Proposition~\ref{prop:mirrorprox} which states that relative Lipschitzness only must hold with respect to triples of algorithm iterates.

\begin{corollary}\label{corr:localrl}
Suppose Algorithm~\ref{alg:mp} is run with a monotone operator $g$ and a distance generating $r$ satisfying, for all iterations $t$,
\begin{equation}\label{eq:localrl}\inprod{g(w_t) - g(z_t)}{w_t - z_{t + 1}} \le \lam\Par{V_{z_t}^r(w_t) + V_{w_t}(z_{t + 1})}.\end{equation}
Then, the conclusion of Proposition~\ref{prop:mirrorprox} holds.
\end{corollary}
\begin{proof}
Observe that the only applications of relative Lipschitzness in the proof of Proposition~\ref{prop:mirrorprox} are of the form \eqref{eq:localrl} (namely, in \eqref{eq:telescope}). Thus, the same conclusion still holds.
\end{proof}

 We use Corollary~\ref{corr:localrl} to give an alternative algorithm and analysis recovering the rates implied by the use of area convexity in \cite{Sherman17}, for box-simplex games, which we now define.
 
 \begin{problem}[Box-simplex game]\label{prob:boxlinf}
 Let $A \in \R^{m \times n}$ be a matrix and let $b \in \R^m$, $c \in \R^n$ be vectors. The associated box-simplex game, and its induced monotone operator $g$, are
 \begin{equation}\label{eq:boxsimp}\min_{x \in [-1, 1]^n} \max_{y \in \Delta^m} f(x,y) \defeq y^\top Ax - \inprod{b}{y} + \inprod{c}{x},\; g(x, y) \defeq \Par{A^\top y + c, b - Ax}.\end{equation}
 Here, $\Delta^m \defeq \{y \in \R^m_{\geq 0} : \sum_{i \in [m]} y_i = 1 \}$ is the nonnegative probability simplex in $m$ dimensions.
 \end{problem}

By a simple reduction that at most doubles the size of the input (stacking $A$, $b$ with negated copies, cf.\ Section 3.1 of \cite{SidfordT18}), Problem~\ref{prob:boxlinf} is a generalization of the box-constrained $\ell_\infty$-regression problem
\[\min_{x \in [-1, 1]^m} \norm{Ax - b}_\infty.\]
The work of \cite{Sherman17}  proposed a variant of extragradient algorithms, based on taking primal-dual proximal steps in the following regularizer\footnote{We let $\norm{A}_{\infty \rightarrow \infty} \defeq \sup_{\norm{x}_\infty = 1} \norm{Ax}_\infty$, i.e.\ the $\ell_\infty$ operator norm of $A$ or max $\ell_1$ norm of any row.}:
\begin{equation}
\label{eq:shermanreg}
r(x, y) \defeq y^\top |A| (x^2) + 10 \norm{A}_{\infty \rightarrow \infty}\sum_{i \in [m]} y_i \log y_i.
\end{equation}
Here, $|A|$ is the entrywise absolute value of $A$. The convergence rate of this algorithm was proven in \cite{Sherman17} via an analysis based on ``area convexity'' of the pair $(g, r)$, which required a somewhat sophisticated proof based on solving a partial differential equation over a triangle. We now show that the same rate can be obtained by the extragradient algorithms of \cite{Nemirovski04, Nesterov07}, and analyzed via local relative Lipschitzness \eqref{eq:localrl}\footnote{Although our analysis suffices to recover the rate of \cite{Sherman17} for $\ell_\infty$ regression, the analysis of \cite{Sherman17} is in some sense more robust (and possibly) more broadly applicable than ours, as it does not need to reason directly about how much the iterates vary in a step. Understanding or closing this gap is an interesting open problem.}. We first make the following simplication without loss of generality.

\begin{lemma}\label{lem:preprocess}
For all  $x \in [-1,1]^n$ the value of $\max_{y \in \Delta^m} f(x,y)$ in \eqref{eq:boxsimp} is unchanged if we remove all coordinates of $b$ with $b_i \ge \min_{i^* \in [m]} b_{i^*} + 2\norm{A}_{\infty \rightarrow \infty}$, and the corresponding rows of $A$. Therefore, in designing an algorithm to solve \eqref{eq:boxsimp} to additive error with linear pre-processing it suffices to assume that $b_i \in [0, 2\norm{A}_{\infty \rightarrow \infty}]$ for all $i \in [m]$.
\end{lemma}
\begin{proof}
For any $x \in [-1,1]^n$, letting $i^* \in \argmin_{i \in [m]} b_i$ we have 
\[
\max_{y \in \Delta^m} y^\top\Par{Ax - b} = \max_{i \in [m]} [A x - b]_{i}
\geq -\norm{A}_{\infty \rightarrow \infty} \norm{x}_\infty - \min_{i^* \in [m]} b_{i}
\geq  -\norm{A}_{\infty \rightarrow \infty}  - b_{i^*} ~.
\]	
However, $[Ax - b]_i \leq \norm{A}_{\infty \rightarrow \infty} - b_i$ for all $i \in [m]$. Consequently, any coordinate $i \in [m]$ that satisfies $b_i \ge b_{i^*} + 2\norm{A}_{\infty \rightarrow \infty}$ has $[Ax - b]_i \leq [Ax - b]_{i_*}$ and the value of $\max_{y \in \Delta^m} f(x,y)$ is unchanged if this entry of $b_i$ and the corresponding row of $A$ is removed. Further, note that $\inprod{y}{\mathbf{1}}$ is a constant for all $y \in \Delta^m$. Consequently, in linear time we can remove all the coordinates $i$ with  $b_i \ge \min_{i^* \in [m]} b_{i^*} + 2\norm{A}_{\infty \rightarrow \infty}$ and shift all the coordinates by an additive constant so that the minimum coordinate of a remaining $b_i$ is $0$ without affecting additive error of any $x$. 
\end{proof}

We now prove our main result regarding the use of mirror prox to solve box-simplex games, using the area convex regularizer analyzed (with a slightly different algorithm) in \cite{Sherman17}.

\begin{theorem}\label{thm:ac}
Assume the preprocessing of Lemma~\ref{lem:preprocess} so that $b\in [0, 2\norm{A}_{\infty \rightarrow \infty}]^m$. Consider running Algorithm~\ref{alg:mp} or Algorithm~\ref{alg:dualex} on the operator in \eqref{eq:boxsimp} with $\lam = 3$, using the regularizer in \eqref{eq:shermanreg}. The resulting iterates satisfy \eqref{eq:localrl}, and thus satisfy the conclusion of Proposition~\ref{prop:mirrorprox}. 
\end{theorem}
\begin{proof}
Fix a particular iteration $t$. We first claim that the simplex variables $w_t^y$ and $z_{t + 1}^y$ obey the following multiplicative stability property: entrywise,
\begin{equation}\label{eq:multstable} w_t^y, z_{t + 1}^y \in \Brack{\half z_t^y, 2 z_t^y}.\end{equation}
We will give the proof for $w_t^y$ as the proof for $z_{t + 1}^y$ follows from the same reasoning. Recall that
\[w_t = \argmin_{w \in \Delta^n \times [-1, 1]^m} \Brace{ \inprod{\frac{1}{\lam} g(z_t)}{w} + V^r_{z_t}(w)},\]
and therefore, defining $\Par{x}^2$ and $\Par{z_t^x}^2$ as the entrywise square of these vectors,
\begin{align*}w_t^y = \argmin_{y \in \Delta^m} \inprod{\gamma_t^y}{y} +  10\norm{A}_{\infty \rightarrow \infty}\sum_{i \in [m]} y_i \log \frac{y_i}{[z_t^y]_i} 
\text{where } \gamma_t^y \defeq \frac{1}{\lam}\Par{b - Az_t^x} + |A|
\left[
\Par{x}^2 - \Par{z_t^x}^2
\right]. 
\end{align*}
Consequently, applying $\log$ and $\exp$ entrywise we have
\[
w_t^y \propto \exp\Par{\log z_t^y - \frac{1}{10\norm{A}_{\infty \rightarrow \infty}}\gamma_t^y} ~.
\]
This implies the desired \eqref{eq:multstable}, where we use that $\norm{\gamma_t^y}_\infty \le 3\norm{A}_{\infty \rightarrow \infty}$, and $\exp(0.6) \le 2$. Next, we have by a straightforward calculation (Lemma 3.4, \cite{Sherman17} or Lemma 6, \cite{JambulapatiST19}) that 
\begin{equation}\label{eq:diagapprox}\nabla^2 r(x, y) \succeq \begin{pmatrix} \diag{|A_{:j}|^\top y} & 0 \\ 0 & \norm{A}_{\infty \rightarrow \infty}\diag{\frac{1}{y_i}}\end{pmatrix}.\end{equation}
By expanding the definition of Bregman divergence, we have
\[V_{z_t}^r(w_t) = \int_0^1 \int_0^\alpha \norm{w_t - z_t}^2_{\nabla^2 r(z_t + \beta (w_t - z_t))} d\beta d\alpha.\]
Fix some $\beta \in [0, 1]$, and let $z_\beta \defeq z_t + \beta (w_t - z_t)$. Since the coordinates of $z_\beta$ also satisfy the stability property \eqref{eq:multstable}, by the lower bound of \eqref{eq:diagapprox}, we have
\begin{align*}
\norm{w_t - z_t}^2_{\nabla^2 r(z_\beta)} &\ge \sum_{i \in [m], j \in [n]} |A_{ij}| \Par{[z_\beta^y]_i\Brack{w_t^x - z_t^x}_j^2 + \frac{1}{[z_\beta^y]_i}\Brack{w_t^y - z_t^y}_i^2} \\
&\ge \half \sum_{i \in [m], j \in [n]} |A_{ij}| \Par{[z_t^y]_i\Brack{w_t^x - z_t^x}_j^2 + \frac{1}{[z_t^y]_i}\Brack{w_t^y - z_t^y}_i^2}.
\end{align*}
By using a similar calculation to lower bound $V_{w_t}^r(z_{t + 1})$, we have by Young's inequality the desired
\begin{align*}
V_{z_t}^r(w_t) + V_{w_t}^r(z_{t + 1}) &\ge \frac 1 4 \sum_{i \in [m], j \in [n]} |A_{ij}| \Par{[z_t^y]_i\Brack{w_t^x - z_t^x}_j^2 + \frac{1}{[z_t^y]_i}\Brack{w_t^y - z_t^y}_i^2} \\
&+ \frac 1 4 \sum_{i \in [m], j \in [n]} |A_{ij}| \Par{[z_t^y]_i\Brack{w_t^x - z_{t + 1}^x}_j^2 + \frac{1}{[z_t^y]_i}\Brack{w_t^y - z_{t + 1}^y}_i^2} \\
&\ge \frac 1 3 \sum_{i \in [m], j \in [n]} A_{ij} \Par{\Brack{w_t^y - z_t^y}_i\Brack{w_t^x - z_{t + 1}^x}_j - \Brack{w_t^y - z_{t + 1}^y}_i\Brack{w_t^x - z_t^x}_j} \\
&= \frac{1}{\lam} \inprod{g(w_t) - g(z_t)}{w_t - z_{t + 1}}.
\end{align*}
\end{proof}

The range of the regularizer $r$ is bounded by $O(\norm{A}_{\infty \rightarrow \infty} \log m)$, and hence the iteration complexity to find an $\eps$ additively-approximate solution to the box-simplex game is $O(\tfrac{\norm{A}_{\infty \rightarrow \infty} \log m}{\eps})$. Finally, we comment that the iteration complexity of solving the subproblems required by extragradient methods in the regularizer $r$ to sufficiently high accuracy is logarithmically bounded in problem parameters via a simple alternating minimization scheme proposed by \cite{Sherman17}. Here, we note that the error guarantee e.g.\ Proposition~\ref{prop:mirrorprox} is robust up to constant factors to solving each subproblem to $\eps$ additive accuracy, and appropriately using approximate optimality conditions (for an example of this straightforward extension, see Corollary 1 of \cite{JambulapatiST19}). %
\section{Randomized coordinate acceleration via expected relative Lipschitzness}
\label{sec:coordinate}

We show relative Lipschitzness can compose with randomization. Specifically, we adapt Algorithm~\ref{alg:accel_basic} to coordinate smoothness, recovering the accelerated rate first obtained in \cite{ZhuQRY16, NesterovS17}. We recall $f$ is $L_i$-coordinate-smooth if its coordinate restriction is smooth, i.e.\ $|\nabla_i f(x + ce_i) - \nabla_i f(x)| \leq L_i |c|$ $\forall x \in \xset,\; c \in \mathbb{R}$; for twice-differentiable coordinate smooth $f$, $\nabla^2_{ii} f(x) \leq L_i$.

Along the way, we build a framework for randomized extragradient methods via ``local variance reduction" in Proposition~\ref{prop:generalmp}. In particular, we demonstrate how for separable domains our technique can yield $O(T^{-1})$ rates for stochastic extragradient algorithms, bypassing a variance barrier encountered by prior methods \cite{JuditskyNT11}. Throughout, let
$f: \R^d \rightarrow \R$ be $L_i$-smooth in coordinate $i$, and $\mu$-strongly convex in $\norm{\cdot}_2$, and define the distance generating function $r(x, y) = \tfrac{\mu}{2}\norm{x}_2^2 + f^*(y)$.

Our approach modifies that of Section~\ref{sec:acceleration} in the following ways. First, our iterates are defined via stochastic estimators which ``share randomness'' (use the same coordinate in both updates). Concretely, fix some iterate $z_t = (x_t, \nabla f(v_t))$. For a distribution $\{p_i\}_{i \in [d]}$, sample $i \sim p_i$ and let
\begin{equation}
\label{eq:coordmp}
\begin{aligned}
&g_i(z_t) \defeq \left(\frac{1}{p_i} \nabla_i f(v_t), v_t - x_t \right),\;\wi = \left(\hxi, \nabla f(\hvi)\right) \defeq \textrm{Prox}^r_{z_t}\left(\frac{1}{\lam} g_i(z_t)\right), \\
&g_i(\wi) \defeq \left(\frac{1}{p_i} \nabla_i f(\hvi), \hvi - \left(x_t + \frac{1}{p_i} \di\right) \right) \text{ for } \di \defeq \hxi - x_t, \\
&\pzi = \left(\pxi, \nabla f(\pvi)\right) \defeq \textrm{Prox}^r_{z_t}\left(\frac{1}{\lam} g_i(\wi)\right).
\end{aligned}
\end{equation}
By observation, $g_i(z_t)$ is unbiased for $g(z_t)$; however, the same cannot be said for $g_i(\wi)$, as the random coordinate was used in the definition of $\wi$. Nonetheless, examining the proof of Proposition~\ref{prop:mirrorprox}, we see that the conclusion
\[
\inprod{g(\bar{w}_t)}{\bar{w}_t - u} \le V_{z_t}^r(u) - \E\left[V^r_{\pzi}(u)\right]
\]
still holds for some point $\bar{w}_t$, as long as
\begin{equation}\label{eq:randomconds}\begin{aligned}\E\left[\inprod{g_i(\wi)}{\wi - u}\right] &= \inprod{g(\bar{w}_t)}{\bar{w}_t - u},\\
\E\left[\inprod{g_i(\wi) - g_i\left(z_t\right)}{\wi - \pzi}\right]&\le \lam\E\left[V^r_{z_t}(\wi) + V^r_{\wi}(\pzi)\right].\end{aligned}\end{equation}
We make this concrete in the following claim, a generalization of Proposition~\ref{prop:mirrorprox} which handles randomized operator estimates as well as an expected variant of relative Lipschitzness. We remark that as in Corollary~\ref{corr:localrl}, the second condition in \eqref{eq:randomconds} only requires relative Lipschitzness to hold for the iterates of the algorithm, rather than globally.
\begin{proposition}\label{prop:generalmp}
Suppose in every iteration of Algorithm~\ref{alg:mp}, steps are conducted with respect to randomized gradient operators $\Brace{g_i(z_t), g_i(w_t^{(i)})}$ satisfying \eqref{eq:randomconds} for some $\{\bar{w}_t\}$. Then, for all $u \in \zset$,
\[\E\Brack{\sum_{0 \le t < T} \inprod{g\Par{\bar{w}_t}}{\bar{w}_t - u}} \le \lam V_{z_0}^r(u).\]
\end{proposition}
\begin{proof}
The proof follows identically to that of Proposition~\ref{prop:mirrorprox}, where we iterate taking expectations over \eqref{eq:telescope}, each time applying the two conditions in \eqref{eq:randomconds}.
\end{proof}

For the rest of this section, we overload $g_i$ to mean the choices used in \eqref{eq:coordmp}. This choice is motivated via the following two properties, required by \eqref{eq:randomconds} (and shown in Appendix~\ref{app:itermaint}).
\begin{restatable}{lemma}{restateavgmakessense}
\label{lem:avgmakesense}
Let $\bar{w}_t \defeq (x_t + \sum_{i \in [d]} \di, \nabla f(\hvi))$. Then $\forall u$, taking expectations over iteration $t$,
\[\E\left[\inprod{g_i(\wi)}{\wi - u}\right] = \inprod{g(\bar{w}_t)}{\bar{w}_t - u}.\]
\end{restatable}
\begin{restatable}[Expected relative Lipschitzness]{lemma}{restateexpectlam}
\label{lem:expectlam}
Let $\lam = 1 + S_{1/2}/\sqrt{\mu}$, where $S_{1/2} \defeq \sum_{i \in [d]}\sqrt{L_i}$. Then, for the iterates \eqref{eq:coordmp} with $p_i = \sqrt{L_i}/S_{1/2}$, taking expectations over iteration $t$,
\begin{align*}\E\left[\inprod{g_i(\wi) - g_i\left(z_t\right)}{\wi - \pzi}\right] 
\le \lam\E\left[V^r_{z_t}(\wi) + V^r_{\wi}(\pzi)\right].\end{align*}
\end{restatable}

Crucially, our proof of these results uses the fact that our randomized gradient estimators are $1$-sparse in the $x$ component, and the fact that we ``shared randomness'' in the definition of the gradient estimators. Moreover, our iterates are efficiently implementable, under the ``generalized partial derivative oracle'' of prior work \cite{LeeS13, ZhuQRY16, NesterovS17}, which computes $\nabla_i f(ax + by)$ for $x, y \in \mathbb{R}^d$ and $a, b \in \mathbb{R}$. In many settings of interest, these oracles can be implemented with a dimension-independent runtime; we defer a discussion to previous references.

\begin{lemma}[Iterate maintenance]
\label{lem:implinformal}
We can implement each iteration of Algorithm~\ref{alg:coord} using two generalized partial derivative oracle queries and constant additional work.
\end{lemma}

We defer a formal statement to Appendix~\ref{app:itermaint}, as Lemma~\ref{lem:impl}. Combining Lemma~\ref{lem:avgmakesense} and Lemma~\ref{lem:expectlam},  \eqref{eq:randomconds} is satisfied with $\lam = 1 + S_{1/2}/\sqrt{\mu}$. Finally, all of these pieces directly imply the following, via the proof of Theorem~\ref{thm:accel} and iterating expectations. We give our full method as Algorithm~\ref{alg:coord}. 

\begin{theorem}[Coordinate acceleration]
\label{thm:coordaccel}
Algorithm~\ref{alg:coord} produces an $\eps$-approximate minimizer of $f$ in
\[O\left(
\sum_{i \in [d]}
\sqrt{\frac{L_i}{\mu}} \log\left(\frac{f(x_0) - f(x^*)}{\eps}\right)\right) \text{ iterations in expectation,}\]
with iteration complexity given by Lemma~\ref{lem:implinformal}.
\end{theorem}
\begin{proof}
This follows from the proof of Theorem~\ref{thm:accel}, using Proposition~\ref{prop:generalmp} in place of Proposition~\ref{prop:mirrorprox}.
\end{proof} %
\section{Discussion}
We give a general condition for extragradient algorithms to converge at a $O(T^{-1})$ rate. In turn, we show that this condition (coupled with additional tools such as locality, randomization, or strong monotonicity) yields a recipe for tighter convergence guarantees in structured instances. While our condition applies generally, we find it interesting to broaden the types of instances where it obtains improved runtimes by formulating appropriate VI problems. For example, can we recover acceleration in settings such as finite-sum convex optimization (i.e.\ for stochastic gradient methods) \cite{Zhu17} or composite optimization \cite{BeckT09}? Moreover, we are interested in the interplay between (tighter analyses of) extragradient algorithms with other algorithmic frameworks. For example, is there a way to interpolate between our minimax algorithm and the momentum-based framework of \cite{LinJJ20} to obtain tight runtimes for minimax optimization? Ultimately, our hope is that our methods serve as an important stepping stone towards developing the toolkit for solving e.g.\ convex-concave games and variational inequalities in general. \subsection*{Acknowledgments}

The existence of an extragradient algorithm in the primal-dual formulation of smooth minimization directly achieving accelerated rates is due to discussions with the first author, Michael B. Cohen. The second and third authors are indebted to him, and this work is dedicated in his memory.

We also thank our collaborators in concurrent works, Yair Carmon and Yujia Jin, for many helpful and encouraging conversations throughout the duration of this project, Jelena Diakonikolas, Jonathan Kelner, and Jonah Sherman for helpful conversations, and anonymous reviewers for multiple helpful comments on earlier versions of the paper.

Researchers supported in part by Microsoft Research Faculty Fellowship, NSF CAREER Award CCF-1844855, NSF Grant CCF-1955039, a PayPal research gift, and a Sloan Research Fellowship.
\newpage
\bibliographystyle{alpha}		
\bibliography{primaldual}
\newpage
\begin{appendix}	
\section{Preliminaries}
\label{appendix:convex_analysis}

We use the standard notion of a convex conjugate of a convex function, also known as the Fenchel dual, which we will refer to as a ``dual function'' throughout. 
\begin{definition}[Fenchel dual]
	For convex function $f$, its \emph{Fenchel dual} is defined by
	\begin{equation*}
	f^*(y) \defeq \textup{min}_x \left\{\inprod{y}{x} - f(x)\right\}
	\end{equation*}
\end{definition}

We state several key facts about dual functions that we will frequently make use of. These are well-known and we defer proofs to standard texts, e.g.\ \cite{Rockafellar70}.

\begin{lemma}[Convex conjugate, Fenchel duality]
	For any convex $f$, its Fenchel dual $f^*$, and $x, y$ in their respective domains, $f(x) + f^*(y) \geq \inprod{y}{x}$.
\end{lemma}

\begin{lemma}[Conjugate of a conjugate, maximizing argument]
\label{lem:argmaxconj}
	For convex $f$, it holds that $f^{**}(x) = f(x)$, and if $x = \textup{argmax}_{x'} \{\inprod{y}{x} - f(x')\}$, then $x \in \partial f^*(y)$ where $\partial$ is the subgradient operator. Consequently, if $f$ and $f^*$ are both differentiable, $\nabla f$ and $\nabla f^*$ are inverse functions.
\end{lemma}

\begin{lemma}[Divergence of $f^*$]
	\label{lem:bregofdual}
	The Bregman divergences of dual functions are related by
	\[V^{f^*}_{\nabla f(x)}(\nabla f(x')) = V^f_{x'}(x).\]
\end{lemma}

We prove two statements regarding strong convexity of the dual of a smooth function. The first has been previously been observed by e.g.\ \cite{KakadeST09}, but we include a proof for completeness as a precursor to Lemma~\ref{lem:coordsmoothsc}, a coordinate smoothness generalization.

\begin{lemma}[Strong convexity of the dual]
	\label{lem:scdual}
	Suppose $f$ is convex and $L$-smooth with respect to $\norm{\cdot}$. Then, $f^*$ is $\frac{1}{L}$-strongly convex with respect to $\norm{\cdot}_*$.
\end{lemma}
\begin{proof}
This is equivalent to showing that, for two dual points $\xi_1, \xi_2$, we have
\begin{equation*}
f^*(\xi_2) - f^*(\xi_1) - \inprod{\nabla f^*(\xi_1)}{\xi_1 - \xi_2} \geq \frac{1}{2L} \norm{\xi_1 - \xi_2}_*^2.
\end{equation*}
Writing $y = \nabla f^*(\xi_1), x = \nabla f^*(\xi_2)$, we also have $\xi_1 = \nabla f(y), \xi_2 = \nabla f(x)$ by the earlier analysis of maximizing arguments. Furthermore, by Lemma~\ref{lem:bregofdual} relating the Bregman divergences of conjugate functions, it suffices to show that
\begin{equation*}
f(y) - f(x) - \inprod{\nabla f(x)}{y - x} \geq \frac{1}{2L}\norm{\nabla f(y) - \nabla f(x)}_*^2.
\end{equation*}
Let
\begin{align*}
z = \textrm{argmin}_w \inprod{\nabla f(y) - \nabla f(x)}{w - y} + \frac{L}{2} \norm{w - y}^2,
\end{align*}
where we note that by the equality case of Cauchy-Schwarz, we have
\begin{align*}
\inprod{\nabla f(y) - \nabla f(x)}{z - y} + \frac{L}{2} \norm{z - y}^2 = -\frac{1}{2L} \norm{\nabla f(y) - \nabla f(x)}_*^2.
\end{align*}
Then, the proof follows directly by using convexity and smoothness. Indeed,
\begin{align*}
0 &\leq f(z) - f(x) - \inprod{\nabla f(x)}{z - x} \\
&\leq f(y) + \inprod{\nabla f(y)}{z - y} + \frac{L}{2} \norm{z - y}^2 - f(x) - \inprod{\nabla f(x)}{z - x} \\
& = f(y) - f(x) - \inprod{\nabla f(x)}{y-x} + \inprod{\nabla f(y) - \nabla f(x)}{z - y} + \frac{L}{2} \norm{z - y}^2 \\
&= f(y) - f(x) - \inprod{\nabla f(x)}{y - x} - \frac{1}{2L}\norm{\nabla f(y) - \nabla f(x)}_*^2,
\end{align*}
where in the last line, we used the definition of $z$. This proves the desired claim.
\end{proof}

The following Lemma~\ref{lem:coordsmoothsc} can be thought of as a limit of Lemma~\ref{lem:scdual} in diagonal quadratic norms $\norm{\cdot}$ where all but one coordinate tends to $\infty$ (and in the dual, only one coordinate is nonzero), which captures coordinate smoothness. We provide a more direct proof for completeness.

\begin{lemma}[Coordinate strong convexity of the dual]
	\label{lem:coordsmoothsc}
	For a convex function $f$ which is $L_i$-smooth in the $i^{th}$ coordinate, we have
	\begin{align*}
	f(y) \geq f(x) + \inprod{\nabla f(x)}{y - x} + \frac{1}{2L_i} \left|\nabla_i f(y) - \nabla_i f(x)\right|^2.
	\end{align*}
\end{lemma}
\begin{proof}
The proof is essentially the same as that of Lemma~\ref{lem:scdual}, with a tighter guarantee given by coordinate smoothness. Let $\norm{\cdot}$ be the $\ell_2$ norm, and
\begin{align*}
z = \textrm{argmin}_{w: w = y + te_i} \inprod{\nabla f(y) - \nabla f(x)}{w - y} + \frac{L_i}{2} \norm{w - y}^2
\end{align*}
where we note that by the equality case of Cauchy-Schwarz, and as $y$ and $z$ only differ in the $i^{th}$ coordinate, we have
\begin{align*}
\inprod{\nabla f(y) - \nabla f(x)}{z - y} + \frac{L_i}{2} \norm{z - y}^2 = -\frac{1}{2L_i} \left|\nabla_i f(y) - \nabla_i f(x)\right|^2.
\end{align*}
Then, the proof follows:
\begin{align*}
0 &\leq f(z) - f(x) - \inprod{\nabla f(x)}{z - x} \\
&\leq f(y) + \inprod{\nabla f(y)}{z - y} + \frac{L_i}{2} \norm{z - y}^2 - f(x) - \inprod{\nabla f(x)}{z - x} \\
& = f(y) - f(x) - \inprod{\nabla f(x)}{y-x}+ \inprod{\nabla f(y) - \nabla f(x)}{z - y} + \frac{L_i}{2} \norm{z - y}^2 \\
&= f(y) - f(x) - \inprod{\nabla f(x)}{y - x} -\frac{1}{2L_i} \left|\nabla_i f(y) - \nabla_i f(x)\right|^2
\end{align*}
where in the last line, we used the definition of $z$. This proves the desired claim.
\end{proof}

The following property of the Bregman divergence is well-known.

\begin{restatable}{lemma}{restateThreePoint}
	\label{lem:threepoint}
	For any function $r$, $\inprod{-\nabla V^r_a(b)}{b - c} = V^r_a(c) - V^r_b(c) - V^r_a(b)$, $\forall a, b, c$.
\end{restatable}

\begin{proof}
	It suffices to expand via the definition of the Bregman divergence,
	\begin{align*}
	V^r_a(c) - V^r_b(c) - V^r_a(b)	= \inprod{\nabla r(a) - \nabla r(b)}{b - c} = \inprod{-\nabla V^r_a(b)}{b - c}.
	\end{align*}
\end{proof}

As a corollary, we obtain the following.

\begin{restatable}{lemma}{restateProxMinimizer}
	\label{lem:proxprop}
	For $w = \textup{Prox}^r_z(g)$, $\forall u$, $\inprod{g}{w - u} \leq V^r_z(u) - V^r_w(u) - V^r_z(w)$.
\end{restatable}

\begin{proof}
	By first order optimality of the function $\inprod{g}{w} + V^r_z(w)$ as a function of $w$, we have
	\begin{align*}
	\inprod{g + \nabla V^r_z(w)}{u - w} \geq 0 \Rightarrow \inprod{g}{w - u} \leq \inprod{- \nabla V^r_z(w)}{w - u} = V^r_z(u) - V^r_w(u) - V^r_z(w),
	\end{align*}
	where we used the previous lemma.
\end{proof}

\restatebasiclam*
\begin{proof}
	By Cauchy-Schwarz, Lipschitzness of $g$, and strong convexity of $r$,
	\begin{align*}
	\inprod{g(w) - g(z)}{w - u} &\le \norm{g(w) - g(z)}_*\norm{w - u} \le L\norm{w - z} \norm{w - u} \\
	&\le \frac{L}{2}\Par{\norm{w - z}^2 + \norm{w - u}^2} \le \frac{L}{\mu}\left(V^r_z(w) + V^r_w(u)\right).
	\end{align*}
\end{proof}

\restaterelsmooth*
\begin{proof}
	By assumption of relative smoothness of $f$ and the definition of divergence,
	\begin{align*}
	L \left( V^r_z(w) + V_w^r(u) \right) 
	&\geq V^f_z(w) + V_w^f(u) \\
	&= f(w) - \left[ f(z) + \nabla f(z)^\top (w - z) \right]
	+ f(u) - \left[ f(w) + \nabla f(w)^\top (u - w)  \right] \\
	&= V_z^f(u) - \nabla f(z)^\top (z - u) - \nabla f(z)^\top(w - z) + \nabla f(w)^\top (w - u) \\
	&= V_z^f(u) + \langle g(w) - g(z) , u - z \rangle ~.
	\end{align*}
	The result follows from the fact that $ V_z^f(u) \geq 0$ by convexity of $f$.
\end{proof} 	%

\section{Unaccelerated smooth convex optimization via mirror prox}
\label{sec:naive_smooth_rate}

For completeness, we provide here the proof that the $1/T$ rate of extragradient methods for solving Lipschitz monotone VIs also yields a $1/T$ rate for smooth function minimization. Given a smooth, convex $f$ this rate is achieved by applying these methods to solve the VI induced by $g(x) \defeq \nabla f(x)$. Consequently, this section highlights that the accelerated rates we achieve for minimizing $f$ are via working in the expanded primal-dual space of the VI induced by the Fenchel game, as well as our fine-grained notion of relative Lipschitzness, as considered in \Cref{sec:acceleration}.

\begin{lemma}
For $L$-smooth, convex $f : \R^d \rightarrow \R$ define $g : \R^d \rightarrow \R$ by $g(x) \defeq \nabla f(x)$, and for arbitrary $x_0$ let $r(x) \defeq \frac{1}{2} \norm{x_0 - x}_2^2$. For all $T \geq 0$, both $\textsc{Mirror-Prox}(x_0, T)$ (\Cref{alg:mp}) and 
$\textsc{Dual-Ex}(x_0, T)$ (\Cref{alg:dualex}) produce $\{w_t\}_{0 \le t < T}$ such that for any $x^*$ minimizing $f$,
\[
f \left(  \frac{1}{T} \sum_{0 \leq t < T} w_t \right) - f(x^*)
\leq \frac{L \norm{x_0 - x^*}_2^2}{2T}.
\]
\end{lemma}

\begin{proof}
Since $x_0 = \argmin_{x \in \R^d} r(x)$, and since (by \Cref{lem:basiclam} and smoothness) $g$ is $L$-relatively Lipschitz with respect to $r$, \Cref{prop:mirrorprox} and \Cref{prop:dualex} imply
\[
\sum_{0 \leq t < T} 
\langle g(w_t), w_t - u \rangle \leq L V_{x_0}^r(u).
\]
Consequently, applying convexity and letting $u = x^*$ above yields
\begin{align*}
f \left(  \frac{1}{T} \sum_{0 \leq t < T} w_t \right) - f(x_*)
&\leq 
\frac{1}{T} \sum_{0 \leq t < T} 
\Par{
f(w_t) - f(x^*)} \leq \frac{1}{T}
\sum_{0 \leq t < T} 
\langle g(w_t), w_t - x^*\rangle 
\leq \frac{LV_{x_0}^r(x^*) }{T} 
\\
&= \frac{L \norm{x_0 - x^*}_2^2}{2T}.
\end{align*}
\end{proof}
We remark that for $\mu$-strongly convex functions, it is well-known that this also implies a $O(\tfrac L \mu)$ rate of linear convergence to an approximate minimizer by converting function error back to a bound on the squared distance to $x^*$ via applying strong convexity. 	%
\section{Minimax optimization}
\label{app:minimax}

In this section, we give a new fine-grained complexity bound of minimax optimization under blockwise strong convexity and smoothnesses of the problem. Specifically, consider the problem
\begin{equation}\label{eq:minimax}\min_{x \in \R^d} \max_{y \in \R^d} f(x, y),\end{equation}
and assume that the objective satisfies the following bounds\footnote{Here, $\norm{\cdot}_{\text{op}}$ is the $\ell_2$ operator norm.}:
\begin{enumerate}
	\item $\norm{\nabla^2_{xx} f(x, y)}_{\text{op}} \le \lxx$, $\norm{\nabla^2_{xy} f(x, y)}_{\text{op}} \le \lxy$, $\norm{\nabla^2_{yy} f(x, y)}_{\text{op}} \le \lyy$ everywhere
	\item $f(x, \cdot)$ is $\muy$-strongly concave for each $x$, and $f(\cdot, y)$ is $\mux$-strongly convex for each $y$
\end{enumerate}
In order to achieve an $\eps$-approximate saddle point to \eqref{eq:minimax}, recent work \cite{ZhangHZ19} shows that
\[\Omega\Par{\sqrt{\frac{\lxx}{\mux} + \frac{\lxy^2}{\mux\muy} + \frac{\lyy}{\muy}} \cdot \log\frac{1}{\eps}}\]
queries to a gradient oracle for \eqref{eq:minimax} are necessary. An upper bound of $\tO\Par{L/\sqrt{\mux\muy}}$ queries was achieved by \cite{LinJJ20}, where $\tO$ hides a polylogarithmic factor in problem parameters, and $L \defeq \max\Par{\lxx, \lxy, \lyy}$. We show an upper bound of 
\[O\Par{\Par{\frac{\lxx}{\mux} + \sqrt{\frac{\lxy^2}{\mux\muy}} + \frac{\lyy}{\muy}} \cdot \log\frac{1}{\eps}}\]
queries via a simple application of our relative Lipschitzness framework, and a known strongly monotone variant of Algorithm~\ref{alg:mp}, for obtaining $\eps$ Euclidean distance to the saddle point of \eqref{eq:minimax}, improving upon the bound of \cite{LinJJ20} in some cases. This can be converted into a duality gap bound while only increasing the logarithmic term by problem parameters; see the end of this section for a discussion. We will use Algorithm~\ref{alg:mp-sm} and a tightening of its guarantees following Definition~\ref{def:spc}.
\begin{algorithm}
	\caption{$\textsc{Mirror-Prox-SM}(z_0, T)$: Strongly monotone mirror prox \cite{CarmonJST19}}
	\begin{algorithmic}\label{alg:mp-sm}
		\STATE \textbf{Input:} Distance generating $r$, $\lam$-relatively Lipschitz, $m$-strongly monotone $g$, initial point $z_0$
		\FOR {$0 \le t < T$}
		\STATE $w_t \gets \Prox^r_{z_t}(\tfrac{1}{\lam}g(z_t))$
		\STATE $z_{t + 1} \gets \argmin_z \{\inprod{\tfrac{1}{\lam}g(w_t)}{z} + V^r_{z_t}(z) + \tfrac{m}{\lam}V^r_{w_t}(z)\}$
		\ENDFOR
	\end{algorithmic}
\end{algorithm}

\begin{restatable}{proposition}{restatempsm}
	\label{prop:mirrorprox-sm}
	The iterates $\{w_t\}$ of Algorithm~\ref{alg:mp-sm} satisfy, for $z^*$ the solution of the VI in $g$, 
	\[V_{z_T}(z^*) \le \Par{1 + \frac{m}{\lam}}^{-T}V^r_{z_0}(z^*).\]
\end{restatable}

Proposition~\ref{prop:mirrorprox-sm} is a strengthening of a known derivation (see e.g.\ Proposition 5 of \cite{CarmonJST19}) under relative Lipschitzness. We defer its proof to Appendix~\ref{appendix:dual_extrapolation}. Our claimed upper bound follows by combining Proposition~\ref{prop:mirrorprox-sm} with the following relative Lipschitzness bound. We note that the same rate can be obtained by directly combining Proposition~\ref{prop:mirrorprox-sm} with a rescaling of the space to make the objective $1$-strongly convex in each variable, but we include a proof via relative Lipschitzness because it makes the calculation more mechanical (and also originally motivated this observation).

\begin{restatable}{lemma}{restatelxymuxy}
	Let $r(x, y) = \tfrac{\mux}{2}\norm{x}_2^2 + \tfrac{\muy}{2}\norm{y}_2^2$ and $g(x, y) = (\nabla_x f(x, y), -\nabla_y f(x, y))$. Then with respect to $r$, $g$ is $1$-strongly monotone and 
	\[\frac{\lxx}{\mux} + \sqrt{\frac{\lxy^2}{\mux\muy}} + \frac{\lyy}{\muy}\text{-relatively Lipschitz}.\]
\end{restatable}
\begin{proof}
First, we prove strong monotonicity: let $z = (z\x, z\y)$ and $w = (w\x, w\y)$. Moreover, let $\jac$ be the Jacobian of the operator $g$, and note that
\[g(w) - g(z) = \int_0^1 \jac(z_t) (w - z)dt,\text{ where } z_t \defeq (1 - t)z + tw.\]
Thus, integrating and using antisymmetry of the off-diagonal blocks of $\jac$,
\begin{align*}
\inprod{g(w) - g(z)}{w - z} = \int_0^1 (w - z)^\top \jac(z_t)(w - z) dt\\
= \int_0^1 \Par{(w\x - z\x)^\top \nabla^2 f_{xx}(z_t) (w\x - z\x) - (w\y - z\y)^\top \nabla^2 f_{yy}(z_t) (w\y - z\y)}dt \\
\ge \int_0^1\Par{\mux \norm{w\x - z\x}_2^2 + \muy \norm{w\y - z\y}_2^2}dt = V^r_z(w) + V^r_w(z).
\end{align*}
In the only inequality, we used our strong convexity and strong concavity assumptions. Next, we prove the relative Lipschitz bound: consider three points $z$, $w$, and $u = (u\x, u\y)$. We have by triangle inequality and the assumed operator norm bounds
\begin{align*}
g\x(w) - g\x(z) = \int_0^1 \Par{\nabla^2_{xx}f(z_t)(w\x - z\x) + \nabla^2_{xy}f(z_t)(w\y - z\y)}dt\\
\implies \norm{g\x(w) - g\x(z)}_2 \le L_{xx} \norm{w\x - z\x}_2 + L_{xy}\norm{w\y - z\y}_2,\\
g\y(w) - g\y(z) = -\int_0^1 \Par{\nabla^2_{yx}f(z_t)(w\x - z\x) + \nabla^2_{yy}f(z_t)(w\y - z\y)}dt\\
\implies \norm{g\y(w) - g\y(z)}_2 \le L_{xy} \norm{w\x - z\x}_2 + L_{yy}\norm{w\y - z\y}_2.
\end{align*}
Therefore, it follows from Cauchy-Schwarz that
\begin{align*}\inprod{g(w) - g(z)}{w - u} \le L_{xx}\norm{w\x - z\x}_2\norm{w\x - u\x}_2 + L_{xy}\norm{w\y - z\y}_2\norm{w\x - u\x}_2 \\
+ L_{xy}\norm{w\x - z\x}_2\norm{w\y - u\y}_2 + L_{yy}\norm{w\y - z\y}_2\norm{w\y - u\y}_2.\end{align*}
Finally, denoting $r\x(x) \defeq \tfrac{\mux}{2}\norm{x}_2^2$, $r\y(y) \defeq \tfrac{\muy}{2}\norm{y}_2^2$, the conclusion of relative Lipschitzness follows from nonnegativity of divergences and combining the three bounds
\begin{align*}
L_{xx}\norm{w\x - z\x}_2\norm{w\x - u\x}_2 &\le \frac{\lxx}{\mux} \Par{V^{r\x}_{z\x}(w\x) + V^{r\x}_{w\x}(u\x)},\\
L_{yy}\norm{w\y - z\y}_2\norm{w\y - u\y}_2 &\le \frac{\lyy}{\muy} \Par{V^{r\y}_{z\y}(w\y) + V^{r\y}_{w\y}(u\y)},\\
L_{xy}\norm{w\y - z\y}_2\norm{w\x - u\x}_2 
+ L_{xy}\norm{w\x - z\x}_2\norm{w\y - u\y}_2 &\le \frac{\lxy}{\sqrt{\mux\muy}}\Par{V^r_z(w) + V^r_w(u)}.
\end{align*}
We prove the last bound here; the other two follow similarly. By expanding the definition of $r$,
\begin{align*}
\frac{\lxy}{\sqrt{\mux\muy}}\Par{V^r_z(w) + V^r_w(u)} &= \frac{\lxy}{\sqrt{\mux\muy}}\left(\frac{\mux}{2}\norm{w\x - z\x}_2^2 + \frac{\muy}{2}\norm{w\y - z\y}_2^2 \right.\\
&\left.+ \frac{\mux}{2}\norm{z_+\x - w\x}_2^2 + \frac{\muy}{2}\norm{z_+\y - w\y}_2^2\right) \\
&\ge \lxy \Par{\norm{w\x - z\x}_2 \norm{w\y - u\y}_2 + \norm{w\x - u\x}_2\norm{w\y - z\y}_2}.
\end{align*}
\end{proof}

We give a brief discussion regarding converting a distance bound from a pair $(x, y)$ to the saddle point $(x^*, y^*)$ of \eqref{eq:minimax} into a duality gap bound. Specifically, let $y'$ be the best response to $x$, and $x'$ be the best response to $y$. Further, denote 
\[h_1(x) \defeq \max_{y \in \R^d} f(x, y),\; h_2(y) \defeq \min_{x \in \R^d} f(x, y).\]
Note the duality gap of the pair $(x, y)$ is precisely
\begin{align*}f(x, y') - f(x', y) &= \Par{f(x, y') - f(x^*, y^*)} + \Par{f(x^*, y^*) - f(x', y)} \\&= \Par{h_1(x) - h_1(x^*)} + \Par{h_2(y^*) - h_2(y)}.\end{align*}
Denote $L = \max(\lxx, \lxy, \lyy)$. Under the setting of this section, it was shown in Lemma B.2 of \cite{LinJJ20} that $h_1$ is $\tfrac{2L^2}{\mux}$ smooth, which implies the bound (since $x^*$ minimizes $h_1$ by definition)
\[h_1(x) - h_1(x^*) \le \frac{L^2}{\mux}\norm{x - x^*}_2^2.\]
Consequently, we can convert a distance bound into a duality gap bound, while only affecting the logarithmic runtime term. A similar argument holds for terms corresponding to $h_2$. 	%
\section{Additional extragradient methods}
\label{appendix:dual_extrapolation}

\subsection{Strongly monotone mirror prox}

We give a proof of Proposition~\ref{prop:mirrorprox-sm}, restated here for convenience.
\restatempsm*
\begin{proof}
	As in \eqref{eq:foopt}, first-order optimality with respect to $z^*$ and nonnegativity of $V_{w_t}(z_{t + 1})$ implies
	\begin{align*}
	\frac{1}{\lam}\inprod{g(z_t)}{w_t - z_{t + 1}} &\leq V^r_{z_t}(z_{t + 1}) - V^r_{w_t}(z_{t + 1}) - V^r_{z_t}(w_t), \\
	\frac{1}{\lam}\inprod{g(w_t)}{z_{t + 1} - z^*} &\leq V^r_{z_t}(z^*) - V^r_{z_{t + 1}}(z^*) - V^r_{z_t}(z_{t + 1}) + \frac{m}{\lam}\Par{V^r_{w_t}(z^*) - V^r_{z_{t + 1}}(z^*)}.
	\end{align*}
	Rearranging and applying relative Lipschitzness (Definition~\ref{def:spc}) as in \eqref{eq:telescope} yields
	\begin{equation}
	\label{eq:linearrate}
	\Par{1 + \frac{m}{\lam}} V_{z_{t + 1}}(z^*) \le \frac{1}{\lam}\Par{\inprod{g(w_t)}{w_t - z^*} - mV_{w_t}(z^*)}+ \Par{1 + \frac{m}{\lam}} V_{z_{t + 1}}(z^*) \leq V_{z_t}(z^*).
	\end{equation}
	Here, we used the fact that by the definition of $z^*$ and strong monotonicity, 
	\[\inprod{g(w_t)}{w_t - z^*} - mV_{w_t}(z^*) \ge \inprod{g(w_t) - g(z^*)}{w_t - z^*} - mV_{w_t}(z^*) \ge 0.\]
	Thus, iterating the inequality \eqref{eq:linearrate} yields the conclusion.
\end{proof}

\subsection{Dual extrapolation}

In this section, we give a simplified presentation of the dual extrapolation algorithm of \cite{Nesterov07} for approximately solving a variational inequality in a monotone operator $g$, using a regularizer $r$. It obtains the same rate of convergence as the mirror prox algorithm, under relative Lipschitzness (Definition~\ref{def:spc}). Specifically, it is the following algorithm which updates a dual variable $s_t$ iteratively.

\begin{algorithm}
	\caption{$\textsc{Dual-Ex}(\bar{z}, T)$: Dual extrapolation \cite{Nesterov07}}
	\begin{algorithmic}\label{alg:dualex}
		\STATE \textbf{Input:} Distance generating $r$, $\lam$-relatively Lipschitz monotone $g: \zset \rightarrow \zset^*$, initial point $\bar{z} \in \zset$
		\STATE$s_0 \gets 0$
		\FOR {$0 \le t < T$}
		\STATE $z_t \gets \Prox^r_{\bar{z}}(s_t)$
		\STATE $w_t \gets \Prox^r_{z_t}(\tfrac{1}{\lam}g(z_t))$
		\STATE $s_{t + 1} \gets s_t + \tfrac{1}{\lam}g(w_t)$
		\ENDFOR
	\end{algorithmic}
\end{algorithm}

\begin{proposition}
	\label{prop:dualex}
	The iterates $\{w_t\}$ of Algorithm~\ref{alg:dualex} satisfy for all $u \in \zset$, 
	\[\sum_{0 \le t < T} \inprod{g(w_t)}{w_t - u} \le \lam V^r_{\bar{z}}(u).\]
\end{proposition}

\Cref{prop:dualex} requires the following helper lemma, the main tool in its proof.

\begin{lemma}
	\label{lem:helpdualex}
	For every iteration $t$,
	\begin{equation*}
	\frac{1}{\lam} \inprod{g(w_t)}{w_t - \bar{z}} \leq \inprod{s_{t + 1}}{z_{t + 1} - \bar{z}} + V^r_{\bar{z}}(z_{t + 1}) - \inprod{s_t}{z_t - \bar{z}} - V^r_{\bar{z}}(z_t).
	\end{equation*}
\end{lemma}

\begin{proof}
We again apply Lemma~\ref{lem:proxprop} on the two steps with respect to $z_{t + 1}$, yielding
\begin{align*}
&\inprod{s_t}{z_t - z_{t + 1}} \leq V^r_{\bar{z}}(z_{t + 1}) - V^r_{z_t}(z_{t + 1}) - V^r_{\bar{z}}(z_t),  \\
&\frac{1}{\lam} \inprod{g(z_t)}{w_t - z_{t + 1}} \leq V^r_{z_t}(z_{t + 1}) - V^r_{w_t}(z_{t + 1}) - V^r_{z_t}(w_t).
\end{align*}
Furthermore, note that by relative Lipschitzness, we have
\begin{align*}
\frac{1}{\lam}\inprod{g(w_t) - g(z_t)}{w_t - z_{t + 1}} \leq V^r_{w_t}(z_{t + 1}) + V^r_{z_t}(w_t).
\end{align*}
Combining these three inequalities and rearranging terms appropriately yields the conclusion:
\begin{align*}
\inprod{s_t}{z_t - z_{t + 1}} + \frac{1}{\lam}\inprod{g(w_t)}{w_t - z_{t + 1}} \leq V^r_{\bar{z}}(z_{t + 1}) - V^r_{\bar{z}}(z_t) \\
\implies\frac{1}{\lam}\inprod{g(w_t)}{w_t - \bar{z}} \leq \inprod{s_{t + 1}}{z_{t + 1} - \bar{z}} + V^r_{\bar{z}}(z_{t + 1}) - \inprod{s_t}{z_t - \bar{z}} - V^r_{\bar{z}}(z_t).
\end{align*}
\end{proof}

This immediately yields the following:
\begin{corollary}
	\label{corr:dualexhelper}
$\Phi_t = \frac{1}{\lam}\sum_{k = 0}^{t - 1} \inprod{g(w_k)}{w_k - \bar{z}} - \inprod{s_t}{z_t - \bar{z}} - V_{\bar{z}}(z_{t})$ is nonincreasing in $t$.
\end{corollary}

\begin{proof}[Proof of Proposition~\ref{prop:dualex}]
Note that for any $u$, 
\begin{align*}
\sum_{t = 0}^{T - 1}\inprod{g(w_t)}{w_t - u} &= \sum_{t = 0}^{T - 1}\inprod{g(w_t)}{w_t - \bar{z}} + \sum_{t = 0}^{T - 1}\inprod{g(w_t)}{\bar{z} - u} + \left(\lam V_{\bar{z}}(u)-\lam V_{\bar{z}}(u)\right) \\
&\leq \sum_{t = 0}^{T - 1}\inprod{g(w_t)}{w_t - \bar{z}} + \sum_{t = 0}^{T - 1}\inprod{g(w_t)}{\bar{z} - z_T} + \left(\lam V_{\bar{z}}(u)-\lam V_{\bar{z}}(z_T)\right) \\
&= \lam\Phi_T + \lam V_{\bar{z}}(u) \leq \lam\Phi_0 + \lam V_{\bar{z}}(u)= \lam V_{\bar{z}}(u).
\end{align*}
The first inequality used the definition of $z_T$, and the second inequality used Corollary~\ref{corr:dualexhelper}.
\end{proof}
We note that all our acceleration results are also implementable with dual extrapolation as the base method, by analogous arguments as in Sections~\ref{sec:acceleration} and~\ref{sec:coordinate}. 	%

\section{Extragradient acceleration in non-Euclidean norms}
\label{appendix:generalize}

In this section, we generalize the developments of Section~\ref{sec:acceleration} to general norms, where $f$ is $L$-smooth in some $\norm{\cdot}$. The notion of strong convexity in general norms is slightly different; for acceleration to be achievable, $f$ must be $\mu$-strongly convex with respect to a regularizer $\omega$, where $\omega$ is $1$-strongly convex in $\norm{\cdot}$; see e.g.\ \cite{Zhu17} for a discussion. 

We first state our general strategy. Let $h(x) \defeq f(x) - \mu\omega(x)$, for all $x \in \R^d$; it is immediate that $h$ is convex and $L$-smooth in $\norm{\cdot}$. To solve the problem $\min_{x \in \R^d} f(x) = \min_{x \in \R^d} h(x) + \mu\omega(x)$, it instead suffices to solve the equivalent saddle point problem
\begin{equation}\label{eq:gennorm_spp} \min_{x \in \R^d} \max_{y \in \R^d} \mu\omega(x) + \inprod{y}{x} - h^*(y).\end{equation}
It is immediate that the saddle point of \eqref{eq:gennorm_spp} is $z^* \defeq (x^*, \nabla h(x^*))$. The key observation is that \eqref{eq:gennorm_spp} is strongly monotone and relatively Lipschitz (with an accelerated parameter) with respect to the natural choice of regularizer,
\begin{equation}\label{eq:rgennorm}r(x,y) \defeq \mu\omega(x) + h^*(y).\end{equation}
From this point, it suffices to apply the strongly monotone extragradient framework of Proposition~\ref{prop:mirrorprox-sm}. We now make this formal by giving the following two helper lemmata.

\begin{lemma}\label{lem:gennorm_sm}
Let $f: \R^d \rightarrow \R$ be $L$-smooth in norm $\norm{\cdot}$, and $\mu$-strongly convex with respect to $\omega$, a 1-strongly convex function in $\norm{\cdot}$. Let $h(x) \defeq f(x) - \mu\omega(x)$ $\forall x \in \R^d$, and let $g(x, y) = \Par{y + \mu\nabla\omega(x), \nabla h^*(y) - x}$ be the gradient operator of the objective \eqref{eq:gennorm_spp}. Then, $g$ is $1$-strongly monotone with respect to the distance generating function $r$ defined in \eqref{eq:rgennorm}.
\end{lemma}
\begin{proof}
By definition (cf.\ Section~\ref{sec:definitions}), it suffices to show that for all $w$, $z \in \R^d \times \R^d$,
\[\inprod{g(w) - g(z)}{w - z} \ge V^r_w(z) + V^r_z(w) = \inprod{\nabla r(w) - \nabla r(z)}{w - z}.\]
By direct expansion, this is an equality. In particular, for $w = (w\x, w\y)$ and $z = (z\x, z\y)$,
\begin{align*}
\inprod{g(w) - g(z)}{w - z} &= \inprod{\Par{w\y - z\y} + \mu\Par{\nabla\omega(w\x) - \nabla\omega(z\x)}}{w\x - z\x} \\
&+ \inprod{\Par{\nabla h^*(w\y) - \nabla h^*(z\y)}- \Par{w\x - z\x}}{w\y - z\y} \\
&= \mu\inprod{\nabla\omega(w\x) - \nabla\omega(z\x)}{w\x - z\x} + \inprod{\nabla h^*(w\y) - \nabla h^*(z\y)}{w\y - z\y} \\
&= \inprod{\nabla r(w) - \nabla r(z)}{w - z}.
\end{align*}
\end{proof}

\begin{lemma}\label{lem:gennorm_rl}
In the setting of Lemma~\ref{lem:gennorm_sm}, $g$ is $1 + \sqrt{\tfrac L \mu}$-relatively Lipschitz with respect to $r$.
\end{lemma}
\begin{proof}
The proof is patterned off of Lemma~\ref{lem:betterlambda}. Consider three points $z = (z\x, z\y)$, $w = (w\x, w\y)$, $u = (u\x, u\y)$. By direct calculation,
\begin{equation}\label{eq:threeterms_gennorm}\begin{aligned}
\inprod{g(w) - g(z)}{w - u} &= \inprod{w\y - z\y}{w\x - u\x} + \inprod{-w\x + z\x}{w\y - u\y} \\
&+ \mu\inprod{\nabla \omega(w\x) - \nabla \omega(z\x)}{w\x - u\x} + \inprod{\nabla h^*(w\y) - \nabla h^*(z\y)}{w\y - u\y}.
\end{aligned}\end{equation}
To bound the first line of \eqref{eq:threeterms_gennorm}, we have the following analog to \eqref{eq:accelcs}, where we use that $h^*$ is $\tfrac 1 L$-strongly convex in $\norm{\cdot}_*$ by Lemma~\ref{lem:scdual}:
\begin{align*}
\inprod{w\y - z\y}{w\x - u\x} + \inprod{z\x - w\x}{w\y - u\y} \le \norm{w\y - z\y}_* \norm{w\x - u\x} + \norm{z\x - w\x} \norm{w\y - u\y}_*\\
\le \sqrt{\frac{L}{\mu}}\Par{\frac{\mu}{2}\norm{w\x - z\x}^2 + \frac{\mu}{2}\norm{w\x - u\x}^2 + \frac{1}{2L}\norm{w\y - z\y}_*^2 + \frac{1}{2L}\norm{w\y - u\y}_*^2} \\
\le \sqrt{\frac{L}{\mu}}\Par{V^r_z(w) + V^r_w(u)}.
\end{align*}
To bound the second line of \eqref{eq:threeterms_gennorm}, we use that the calculation \eqref{eq:threeptconvex} implies that
\begin{align*}
\mu\inprod{\nabla \omega(w\x) - \nabla \omega(z\x)}{w\x - u\x} + \inprod{\nabla h^*(w\y) - \nabla h^*(z\y)}{w\y - u\y} &= \inprod{\nabla r(w) - \nabla r(z)}{w - u} \\
&\le V^r_z(w) + V^r_w(u).
\end{align*}
Combining the above two calculations in the context of \eqref{eq:threeterms_gennorm} yields the desired claim.
\end{proof}

\begin{theorem}\label{thm:acceleration}
In the setting of Lemma~\ref{lem:gennorm_sm}, consider running Algorithm~\ref{alg:mp-sm} on the monotone operator $g$ and the distance generating function $r$ initialized at $z_0 \defeq (x_0, \nabla h(x_0))$, for $T$ iterations. Every iteration consists of solving a constant number of proximal problems \eqref{eq:proxdef} in the function $\omega$, and a constant number of $d$-dimensional vector operations; moreover, we can always maintain each $z_t$ in the form $(x_t, \nabla h(v_t))$ for some explicitly computed $v_t \in \R^d$. Finally, we have
\[T \ge 4\sqrt{\frac L \mu}\log\Par{\frac{2L}{\mu} \cdot \frac{f(x_0) - f(x^*)}{\eps}} \implies f(x_T) - f(x^*) \le \eps.\]
\end{theorem}
\begin{proof}
We first demonstrate the claimed implementability of steps of Algorithm~\ref{alg:mp-sm}. Note that from the form of iterates in Algorithm~\ref{alg:mp-sm}, each $x$ block of variables indeed results from solving proximal problems in $\omega$. On the $y$ block, the claim of the invariant (that it can be represented as some $\nabla h(v)$ for explicit $v$) follows identically to the arguments in Lemma~\ref{lem:iterateform}, where we inductively show that the $y$ block of each $z_t$, $w_t$ can be maintained as a gradient of $h$.

Next, we prove the desired error bound. We first compute, for $z^* \defeq (x^*, \nabla h(x^*))$,
\begin{align*}
V^r_{z_0}(z^*) &= \mu V^\omega_{x_0}(x^*) + V^{h^*}_{\nabla h(x_0)}(\nabla h(x^*)) = \mu V^\omega_{x_0}(x^*) + V^{h}_{x^*}(x_0) \\
&\le V^f_{x_0}(x^*) + V^f_{x^*}(x_0) = \inprod{\nabla f(x_0) - \nabla f(x^*)}{x_0 - x^*} \\
&\le L\norm{x_0 - x^*}^2 \le \frac{2L}{\mu}\Par{f(x_0) - f(x^*)}.
\end{align*}
By applying Proposition~\ref{prop:mirrorprox-sm} with the bounds from Lemmas~\ref{lem:gennorm_sm} and~\ref{lem:gennorm_rl}, we have
\[V^r_{z_T}(z^*) \le \frac{\mu\eps}{L}.\]
The conclusion follows from smoothness of $f$ and strong convexity of $\omega$, i.e.\
\[\frac{L}{\mu}V^r_{z_T}(z^*) \ge \frac{L}{\mu}V^{\mu\omega}_{x_T}(x^*) \ge \frac{L}{2}\norm{x_T - x^*}^2 \ge f(x_T) - f(x^*).\]
The constant $4$ is due to $1 + \sqrt{L/\mu} \le 2\sqrt{L/\mu}$, and $\tfrac L \mu$ being squared in the logarithm.
\end{proof} %
\section{Missing proofs from Section~\ref{sec:coordinate}}
\label{app:itermaint}

\restateavgmakessense*
\begin{proof}
	Note $\hvi$ is deterministic regardless of the sampled $i \in [d]$. Expanding for $u = (u\x, u\y)$,
	\begin{equation*}
	\begin{aligned}
	\E\left[\inprod{g_i(\wi)}{\wi - u}\right] = \sum_{i \in [d]} p_i\left(\inprod{\frac{1}{p_i} \nabla_i f(\hvi)}{\hxi - u\x} \right.\\
	\left.+ \inprod{\hvi - \left(x_t + \frac{1}{p_i} \di\right)}{\hy - u\y}\right)\\
	= \inprod{g(\bar{w}_t)}{\bar{w}_t - u}.
	\end{aligned}
	\end{equation*}
	Here, we used the fact that $\nabla_i f(\hvi)$ is 1-sparse.
\end{proof}

\restateexpectlam*
\begin{proof}
	Equivalently, we wish to show that
	\begin{equation*}
	\begin{aligned}
	\E\left[\inprod{g_i(\wi) - g_i\left(z_t\right)}{\wi - \pzi}\right]\le \left(1 + \frac{S_{1/2}}{\sqrt{\mu}}\right) \E\left[V^r_{z_t}(\wi) + V^r_{\wi}(\pzi)\right].
	\end{aligned}
	\end{equation*}
	The proof is patterned from Lemma~\ref{lem:betterlambda}. By direct calculation, the left hand side is
	\begin{equation}
	\label{eq:threetermscoord}
	\begin{aligned}
	&\E\left[\inprod{g_i(\wi) - g_i\left(z_t\right)}{\wi - \pzi}\right] \\
	= &\sum_{i \in [d]} p_i \left( \frac{1}{p_i} \inprod{\nabla_i f(\hvi) - \nabla_i f\left(v_t\right)}{\hxi - \pxi} \right.\\
	+ &\left.\frac{1}{p_i}\inprod{x_t - \hxi}{\nabla_i f(\hvi) - \nabla_i f(\pvi)} \right)\\
	+& \sum_{i \in [d]} p_i \inprod{\hvi - v_t}{\nabla f(\hvi) - \nabla f(\pvi)}.
	\end{aligned}
	\end{equation}
	We first bound the second and third lines of \eqref{eq:threetermscoord}:
	\begin{equation}
	\label{eq:twotermscontrib}
	\begin{aligned}
	\frac{1}{p_i} \left(\inprod{\nabla_i f(\hvi) - \nabla_i f\left(v_t\right)}{\hxi - \pxi} + \inprod{x_t - \hxi}{\nabla_i f(\hvi) - \nabla_i f(\pvi)}\right) \\
	\le \frac{S_{1/2}}{\sqrt{\mu}}\left(\frac{\mu}{2}\norm{\hxi - \pxi}_2^2 + \frac{1}{2L_i}\norm{\nabla_i f(\hvi) - \nabla_i f(\pvi)}_2^2 \right.\\
	\left.+ \frac{\mu}{2} \norm{x_t - \hxi}_2^2+ \frac{1}{2L_i}\norm{\nabla_i f(\hvi) - \nabla_i f\left(v_t\right)}_2^2\right) \\
	\le \frac{S_{1/2}}{\sqrt{\mu}}\left(V^r_{z_t}(\wi) + V^r_{\wi}(\pzi)\right).
	\end{aligned}
	\end{equation}
	The first inequality used the definition $p_i = \sqrt{L_i/S_{1/2}}$ and Cauchy-Schwarz, and the second used strong convexity and Lemma~\ref{lem:coordsmoothsc}. Next, we bound the fourth line of \eqref{eq:threetermscoord}:
	\begin{equation*}
	\begin{aligned}
	\inprod{\hvi - v_t}{\nabla f(\hvi) - \E\left[\nabla f(\pvi)\right]} \\
	\le V^{f^*}_{\nabla f(v_t)}\left(\nabla f(\hvi)\right) + V^{f^*}_{\nabla f(\hvi)}\left( \E\left[\nabla f(\pvi)\right]\right) \\
	\le V^{f^*}_{\nabla f(v_t)}\left(\nabla f(\hvi)\right) + \E\left[V^{f^*}_{\nabla f(\hvi)}\left(\nabla f(\pvi)\right) \right]\\
	\le \E\left[V^r_{z_t}(\wi) + V^r_{\wi}(\pzi)\right].
	\end{aligned}
	\end{equation*}
	The first inequality is \eqref{eq:threeptconvex}, the second is convexity of Bregman divergence, and the third used nonnegativity of $\tfrac{\mu}{2}\norm{\cdot}_2^2$. Combining with an expectation over \eqref{eq:twotermscontrib} yields the claim.
\end{proof}

\begin{algorithm}[ht!]
	\caption{$\textsc{EG-Coord-Accel}(x_0, \eps)$: Extragradient accelerated coordinate-smooth minimization}
	\begin{algorithmic}
		\label{alg:coord}
		\STATE \textbf{Input:} $x_0 \in \R^d$, $f$ $\{L_i\}_{i \in [d]}$-coordinate smooth and $\mu$-s.c. in $\norm{\cdot}_2$, and $\eps_0 \ge f(x_0) - f(x^*)$
		\STATE $\lam \gets 1 + \sum_{i \in [d]} \sqrt{L_i/\mu}$, $T \gets 4\lceil\lam\rceil$, $K \gets \lceil\log_2\tfrac{\eps_0}{\eps}\rceil$,  $\mathbf{A} \gets \begin{pmatrix}
		1 & \frac{1}{\kappa} - \frac{1}{\kappa^2}\\
		0 & 1 - \frac{1}{\kappa} + \frac{1}{\kappa^2}
		\end{pmatrix}$
		\STATE $p_0 \gets x_0$, $q_0 \gets x_0$, $\mathbf{B_0} \gets \mathbf{I}_{2 \times 2}$
		\FOR {$0 \le k < K$}
		\STATE Sample $\tau$ uniformly in $[0, T - 1]$
		\FOR {$0 \le t < \tau$}
		\STATE Sample $i \propto \sqrt{L_i}$
		\STATE Compute $\nabla_i f(v_t)$, $\nabla_i f((1 - \lam^{-1}) v_t + \lam^{-1} x_t)$ via generalized partial derivative oracle
		\STATE $s_t \gets \begin{pmatrix}
		\tfrac{1}{\mu\lam p_i} \nabla_i f((1 - \lam^{-1}) v_t + \lam^{-1} x_t) &
		\tfrac{1}{\mu\lam^2 p_i^2} \nabla_i f(v_t)
		\end{pmatrix}$
		\STATE $\mathbf{B_{t + 1}} \gets \mathbf{B_t}\mathbf{A}$, $\begin{pmatrix}
		p_{t + 1} & q_{t + 1}
		\end{pmatrix} \gets \begin{pmatrix}
		p_t & q_t
		\end{pmatrix} - s_t \mathbf{B_{t + 1}}^{-1}$ 
		\ENDFOR
		\STATE $\mathbf{B_0} \gets \begin{pmatrix}[\mathbf{B_\tau}]_{12}& [\mathbf{B_\tau}]_{12} \\ [\mathbf{B_\tau}]_{22} & [\mathbf{B_\tau}]_{22}\end{pmatrix}$, $p_0 \gets p_\tau$, $q_0 \gets q_\tau$
		\ENDFOR
		\RETURN $[\mathbf{B_\tau}]_{12} p_\tau + [\mathbf{B_\tau}]_{22} q_\tau$
	\end{algorithmic}
\end{algorithm}

\begin{lemma}
	\label{lem:impl}
	Suppose at step $t$, for $\mathbf{B_t} \in \mathbb{R}^{2 \times 2}$, $p_t, q_t \in \mathbb{R}^d$, we maintain
	\begin{align*}
	\begin{pmatrix}
	x_t & v_t
	\end{pmatrix} = 
	\begin{pmatrix}
	p_t & q_t
	\end{pmatrix} \mathbf{B_t}.
	\end{align*}
	Then, for any sampled $i$, we can compute $\mathbf{B_{t + 1}} \in \mathbb{R}^{2 \times 2}$, $p_{t + 1}, q_{t + 1} \in \mathbb{R}^d$ such that
	\begin{align*}
	\begin{pmatrix}
	\px & \pv
	\end{pmatrix} = 
	\begin{pmatrix}
	p_{t + 1} & q_{t + 1}
	\end{pmatrix} \mathbf{B_{t + 1}},
	\end{align*}
	using two generalized partial derivative oracle queries and constant additional work.
\end{lemma}
\begin{proof}
Suppose on iteration $t$ that coordinate $i$ was sampled. In closed form, we have by the definition of gradient estimators \eqref{eq:coordmp} the updates to the $x$ component
\begin{align*}
\hx = x_t - \frac{1}{\mu\lam p_i} \nabla_i f(v_t),\; x_{t + 1} = x_t - \frac{1}{\mu\lam p_i} \nabla_i f(\hv).
\end{align*}
Thus, if we can maintain the implicit representation of $v_t$ and $\hv$, we can compute these updates via a partial derivative oracle. Next, the proof of Lemma~\ref{lem:iterateform} implies in closed form
\begin{align*}\hvi = \Par{1 - \frac{1}{\lam}} v_t + \frac{1}{\lam} x_t,\\
\pv = v_t + \frac{1}{\lam}\Par{x_t + \frac{1}{p_i}\di - \hv} = \Par{1 - \frac{1}{\lam} + \frac{1}{\lam^2}} v_t + \Par{\frac{1}{\lam} - \frac{1}{\lam^2}} x_t - \frac{1}{\mu\lam^2 p_i^2}\nabla_i f(v_t),\end{align*}
where we have already computed $\nabla_i f(v_t)$. Therefore, the update can be written as, for 2-sparse $s_t$,
\begin{equation*}
\begin{aligned}
\begin{pmatrix}
\px & \pv
\end{pmatrix} = 
\begin{pmatrix}
x_{t} & v_{t}
\end{pmatrix} \mathbf{A} - s_t \\
 \textrm{where} \; \mathbf{A} \defeq
\begin{pmatrix}
1 & \frac{1}{\kappa} - \frac{1}{\kappa^2} \\
0 & 1 - \frac{1}{\kappa} + \frac{1}{\kappa^2}
\end{pmatrix},\; s_t \defeq 
\begin{pmatrix}
\frac{1}{\mu\lam p_i} \nabla_i f(\hv) &
\frac{1}{\mu\lam^2 p_i^2} \nabla_i f(v_t)
\end{pmatrix}.
\end{aligned}
\end{equation*}
We can therefore maintain this implicitly via the equivalent step (doing constant extra work),
\begin{equation*}
\mathbf{B_{t + 1}} = \mathbf{B_t}\mathbf{A}, \;
\begin{pmatrix}
p_{t + 1} & q_{t + 1}
\end{pmatrix} = \begin{pmatrix}
p_t & q_{t}
\end{pmatrix} - s_t \mathbf{B_{t + 1}}^{-1}.
\end{equation*}
\end{proof} 	%
\section{Optimism}
\label{appendix:area_convex}

In this section, we discuss the relationship of our condition (relative Lipschitzness) with the condition of ``optimism'' proposed by \cite{RakhlinS13}, which is known to recover the mirror prox algorithm for Lipschitz operators. Specifically, the optimistic mirror descent procedure of \cite{RakhlinS13} iterates the following steps (see Equation 1 of their paper), for strongly convex $r$:
\begin{equation}\label{eq:optmd}w_t \gets \Prox^r_{z_t}\Par{\eta M_t},\; z_{t + 1} \gets \Prox^r_{z_t}\Par{\eta g(w_t)}.\end{equation}
Here, $M_t$ is a vector which ideally ``resembles'' the point $\eta g(w_t)$, capturing a notion of ``predictability.'' Under Lipschitzness of $g$, \cite{RakhlinS13} notes that simply choosing $M_t = g(z_t)$ and $\eta = \tfrac{1}{L}$ yields an algorithm identical to Algorithm~\ref{alg:mp}, where the notion of predictable sequences comes from stability of $g$. Rakhlin and Sridharan prove the following claim about their procedure (Lemma 1, \cite{RakhlinS13}):
\begin{proposition}[Optimistic mirror descent]\label{prop:optimism}
For $g: \zset \rightarrow \zset^*$ and any $u \in \zset$,
\begin{align*}\sum_{0 \le t < T} \inprod{g(w_t)}{w_t - u} &\le \frac{V^r_{z_0}(u)}{\eta} + \sum_{0 \le t < T} \norm{M_t - g(w_t)}_* \norm{w_t - z_{t + 1}}\\
&- \frac{1}{2\eta}\Par{\norm{w_t - z_t}^2 + \norm{w_t - z_{t + 1}}^2}.\end{align*}
\end{proposition}
Indeed, Proposition~\ref{prop:optimism} follows by the proofs we give of Proposition~\ref{prop:mirrorprox} and Lemma~\ref{lem:basiclam}, where we apply Cauchy-Schwarz and strong convexity of $r$. However, crucially our proof in Section~\ref{sec:acceleration} of the accelerated rate bypasses this direct application of Cauchy-Schwarz to the quantity $M_t - g(w_t)$, and couples terms in a way which can yield a tighter relative Lipschitz parameter (see Lemma~\ref{lem:betterlambda}). In this sense, relative Lipschitzness can  sharpen the rate of convergence for optimistic mirror descent achieved by \cite{RakhlinS13}, which was crucial for our improved guarantees. \end{appendix}

\end{document}